\documentclass[leqno, letterpaper, 11pt]{article}
\usepackage{amssymb,enumerate}
\usepackage{amsmath}
\usepackage{amsfonts}
\usepackage{amsthm}
\usepackage[usenames]{color}
\usepackage{graphicx}
\usepackage{verbatim}
\usepackage[reftex]{theoremref}
\usepackage{bbm}
\usepackage{hyperref}
\hypersetup{colorlinks=true,
            citecolor=green!40!black,
            linkcolor=red!20!black,
            urlcolor=blue!40!black,
            filecolor=cyan!30!black}
            
\usepackage{caption}
\usepackage{subcaption}
\usepackage{mathtools}

\usepackage{authblk}
\usepackage{lineno}

\makeatletter
\DeclareRobustCommand\widecheck[1]{{\mathpalette\@widecheck{#1}}}
\def\@widecheck#1#2{%
    \setbox\z@\hbox{\m@th$#1#2$}%
    \setbox\tw@\hbox{\m@th$#1%
       \widehat{%
          \vrule\@width\z@\@height\ht\z@
          \vrule\@height\z@\@width\wd\z@}$}%
    \dp\tw@-\ht\z@
    \@tempdima\ht\z@ \advance\@tempdima2\ht\tw@ \divide\@tempdima\thr@@
    \setbox\tw@\hbox{%
       \raise\@tempdima\hbox{\scalebox{1}[-1]{\lower\@tempdima\box
\tw@}}}%
    {\ooalign{\box\tw@ \cr \box\z@}}}
\makeatother

\usepackage{pgf,tikz}
\usetikzlibrary{through,calc,positioning,decorations.pathreplacing, arrows.meta}

\usetikzlibrary{arrows}

\graphicspath{{figures/}}

\newcommand{\bR}{{\mathbb R}}
\newcommand{\bZ}{{\mathbb Z}}

\newcommand{\cF}{{\mathcal F}}

\newcommand{\cY}{{\mathcal Y}}
\newcommand{\cZ}{{\mathcal Z}}

\newcommand{\cN}{{\mathcal N}}
\newcommand{\cV}{{\mathcal V}}

\newcommand{\cO}{\mathcal O}

\newcommand{\cT}{\mathcal T}

\newcommand{\cS}{\mathcal S}

\newcommand{\0}{{\it 0\/}}

\newcommand{\vr}{\vec r}
\newcommand{\vc}{\vec c}

\newcommand{\ext}{\mathrm{ex}}

\newcounter{tbox}

\makeatletter
\newcommand{\leqnomode}{\tagsleft@true}
\newcommand{\reqnomode}{\tagsleft@false}
\makeatother

\newcommand{\sta}[1]{\vspace*{0.3cm}  \refstepcounter{tbox}\noindent{ 
\parbox{\textwidth}{(\thetbox) \emph{#1}}}\vspace*{0.2cm}}

\newcounter{mycount}

\newtheorem{theorem}{Theorem}[section]
\newtheorem{corollary}[theorem]{Corollary}
\newtheorem{lemma}[theorem]{Lemma}
\newtheorem{prop}[theorem]{Proposition}

\theoremstyle{definition}

\newtheorem{definition}[theorem]{Definition}

\newcounter{example}

\newcounter{open}

\newcounter{figno}

\newcounter{tableno}

\numberwithin{equation}{section}
\numberwithin{figure}{section}
\numberwithin{table}{section}
\numberwithin{example}{section}

\newcommand{\floor}[1]{\left\lfloor#1\right\rfloor}

\newcommand{\Z}{\cZ}

\newcommand{\row}{\text{\tt row}}
\newcommand{\col}{\text{\tt col}}

\newcommand{\gammap}{\widehat{\gamma}}
\newcommand{\gammad}{\overline{\gamma}}

\definecolor{newc}{rgb}{0.9,0,0}

\title {\bf Young domination on Hamming rectangles}

\author[1]{Janko Gravner}
\author[2]{Matjaž Krnc}
\author[2]{Martin Milanič}
\author[3]{Jean-Florent Raymond}

\affil[1]{Department of Mathematics, University of California, Davis {\tt gravner@math.ucdavis.edu}}
\affil[2]{FAMNIT and IAM, University of Primorska, Koper, Slovenia
{\tt $\{$matjaz.krnc,martin.milanic$\}$@upr.si}}
\affil[3]{Univ.\ Lyon, CNRS, ENS de Lyon, Université Claude Bernard Lyon 1, LIP UMR5668,
  Lyon, France, {\tt jean-florent.raymond@cnrs.fr}}

\date{}
\begin{document}

\maketitle

\begin{abstract}
We introduce a family of domination-type problems in Cartesian products of two graphs. 
The framework captures several well-studied topics, including variants of bootstrap percolation, line growth, distance domination, and target set selection.
We focus on Cartesian products of two complete graphs and formulate 
the notion of \textsl{Young domination number}  
in terms of a growth rule determined by a Young diagram; this number is the smallest cardinality of an initial set that covers the entire vertex set in a prescribed number $L$ of iterations of the rule. 

We compute the Young domination number with $L=1$ for several natural cases, including $k$-domination for Cartesian products of two complete graphs of the same order, thereby proving a conjecture from 2009 due to Burchett, Lane, and Lachniet.
We show that the case of $L=1$ of Young domination is equivalent to computing bipartite Tur\'an numbers for families of double stars, yielding implications of our results in extremal graph theory.
For arbitrary fixed $L$, we devise constant-factor approximation algorithms for the problem.
Our approach is based on a variety of techniques, including duality between Young diagrams, algebraic formulations, explicit constructions, and dynamic programming.
\end{abstract}

\let\thefootnote\relax
\let\thefootnote\relax

\section{Introduction}\label{sec-intro}

\newcommand{\cI}{\mathcal I}
\newcommand{\bN}{\mathbb N}
\newcommand\Osq{\mathbin{\text{\scalebox{.84}{$\square$}}}}

\subsection{Motivation}

Looking for the smallest set of vertices in a graph with a given property is a natural and fundamental problem in extremal combinatorics and algorithmic graph theory. 
Variants of domination in graphs, in particular, have proved useful in many real-life contexts.
We introduce \textsl{Young domination}, a family of domination-type problems in Cartesian products of graphs.  
This framework provides a common generalization of various well-studied problems, including 
bootstrap percolation \cite{GHPS}, 
line growth \cite{GSS} (also known as line percolation \cite{BBLN}), 
distance domination (see, e.g.,~\cite{Hen}), 
$k$-domination on rook's graphs (see~\cite{BLL}, as well as~\cite[Section 8]{HH}), 
and variants of target set selection with constant thresholds (see, e.g.,~\cite{CHLWY,LYW,JW}).
Our problems are formulated in terms of a growth rule determined by a Young diagram, which ensures monotonicity; without this property, 
related
questions become much less tractable (see, e.g., \cite{GG2}).
While we begin with a more general setup, we focus on growth dynamics on Cartesian products of two complete graphs, arguably the simplest nontrivial class.

\subsection{Young domination}

 We now proceed with formal definitions, starting with the growth rule parametrized by a Young diagram. 
The corresponding Young domination number will then be the smallest cardinality of an initial set that covers the entire vertex set in a prescribed number of iterations of the rule.  

Consider two finite graphs $G_1$ and $G_2$. 
The \emph{Cartesian product} of $G_1$ and $G_2$ is the graph $G_1\Box G_2$ with vertex set $V(G_1)\times V(G_2)$ in which two vertices $(u_1,u_2)$ and $(v_1,v_2)$ are adjacent if and only if $u_1 = v_1$ and $u_2v_2\in E(G_2)$ or $u_2 = v_2$ and $u_1v_1\in E(G_1)$.
Each vertex $v=(v_1,v_2)\in V(G_1\Box G_2)$ has neighbors divided into two disjoint sets: $\cN^{r}(v)$, the \emph{row neighborhood} of $v$, defined as the set of all vertices $(v_1,y)$ such that $yv_2$ is an edge of $G_2$, and $\cN^{c}(v)$, the \emph{column neighborhood} of $v$, defined as the set of all vertices $(x,v_2)$ such that $xv_1$ is an edge of~$G_1$.

For an integer $c\in \bZ_+=\{0,1,2,\ldots\}$, we denote by $[0,c]$ the set $\{0,1,\ldots, c\}$.
For integers $a,b\ge 1$, we let $R_{a,b}=[0,a-1]\times [0,b-1]\subseteq\bZ_+^2$ be the discrete $a\times b$ rectangle.
A set $\Z=\bigcup_{(a,b)\in I} R_{a,b}$, given by a union
of rectangles over some set $I\subseteq \bN^2$, is called a (discrete) {\em zero-set\/}. 
We allow the trivial case $\Z=\emptyset$ and also infinite $\cZ$. 
Observe that zero-sets are exactly 
\emph{downward-closed} sets in the sense that if $(x,y)\in \Z$ and $0\le x'\leq x$, $0\le y'\leq y$ then $(x', y') \in \Z$. 
Therefore, zero-sets are \emph{Young diagrams} in the French notation \cite{Rom}.
Each zero-set determines a type of domination 
on Cartesian products of two graphs as follows. 

Given a zero-set $\cZ$, the \emph{transformation} $\cT = \cT(\cZ)$, defined on sets of vertices of $G_1\Box G_2$, assigns to every set $A\subseteq V(G_1\Box G_2)$ a set $
\cT(A)\subseteq V(G_1\Box G_2)$ as follows:
\begin{itemize}
\item the input set is enlarged, i.e., if $v\in A$ then $v\in \cT(A)$;
\item if $v \in V(G_1\Box G_2)\setminus A$, then $v\in \cT(A)$ if and only if $(|\cN^{r}(v)\cap A|,|\cN^{c}(v)\cap A|)\notin\cZ$.
\end{itemize}  
Thus, the iteration $\cT^t(A)$ yields a growth dynamics with 
initial set $A=\cT^0(A)$, and we define 
$\cT^\infty(A)=\bigcup_{t\ge 0} \cT^t(A)$. Observe that
$\cT$ is {\it monotone\/}: $A_1\subseteq A_2$ implies 
$\cT(A_1)\subseteq \cT(A_2)$; in fact, $\cT$ is monotone if and only if
its zero-set $\cZ$ is downward closed. We call the transformation $\cT$, 
or equivalently, the resulting growth dynamics the {\it neighborhood growth\/}.  

For $L=0, 1,2, \ldots,\infty$, we say that $A$ is a \emph{$\cZ$-dominating set with latency $L$} if $\cT^L(A)=V(G_1\Box G_2)$. Our focus is on smallest such sets, 
so we define the \emph{$\cZ$-domination number} of $(\cZ,G_1,G_2)$ \emph{with latency $L$} as
\[
\gamma^L=\gamma^L(\cZ, G_1,G_2)=\min\{|A|\colon \cT^L(A)=V(G_1\Box G_2)\}.
\]
It is trivially true that $\gamma^0(\cZ, G_1,G_2)=|V(G_1\Box G_2)|$.
When the latency is not specified, we assume the default choice $L=1$, 
often omit it from the notation and call $\gamma=\gamma(\cZ, G_1,G_2)=\gamma^1(\cZ, G_1,G_2)$ the \emph{$\cZ$-domination number} of $(\cZ,G_1,G_2)$.   
In particular, if $a\in \mathbb{N}$ and $\cZ$ is the triangle 
\begin{equation}\label{eq-bootstrap}
T_a=\{(x,y)\in \mathbb{Z}_+^2\colon x+y\le a-1\}, 
\end{equation} 
then $\gamma(T_a, G_1,G_2)$ is exactly the $a$-domination number of the product graph $G_1\Box G_2$ \cite{HV}. 
Furthermore, $\gamma^L(T_1, G_1,G_2)$ is exactly the distance-$L$ domination number of the product graph $G_1\Box G_2$, for any finite $L$ (see, e.g.,~\cite{Hen}).
In our general discussions, or if $\cZ$ and $L$ are clear from the context, we refer to $\cZ$-domination with latency $L$ as {\it Young domination\/}.
 
Perhaps the most natural special case of Young domination is where the two factors are complete graphs, $G_1=K_m$ and $G_2=K_n$ on vertex sets $[0,m-1]$ and $[0,n-1]$, in which case we write 
$$
\gamma^L=\gamma^L(\cZ, m,n)=\gamma^L(\cZ, K_m,K_n)\,.
$$
The graph $K_m\Box K_n$ is known as a \emph{Hamming rectangle}, and neighborhood growth on this graph with general $\cZ$ was studied in \cite{GSS, GPS}. 
As in these references, we represent this graph as having vertex set $R_{n,m}$; see the beginning of Section~\ref{sec-prelim} for more details. 
We focus solely on the neighborhood growth dynamics on Hamming rectangles for the remainder of the paper.
See Figure~\ref{fig:intro} for an example.

\subsection{Main results}

We begin with three special cases where Young domination number is given by simple expressions. 
The easiest are the $L$-shaped Young diagrams, given by zero-sets
$V_{a,b}$, where $a,b\ge 0$ and
\begin{equation}\label{eq-L-shape}
    (V_{a,b})^c=\mathbb{Z}_+^2\setminus V_{a,b} = [a,\infty)\times [b,\infty), 
\end{equation}
i.e., a point requires at least $a$ occupied row neighbors \emph{and} at least $b$ occupied column neighbors to become occupied.

\begin{prop}\label{prop-L-shape}
 Assume that $\cZ=V_{a,b}$, where $0\le a\le n$, $0\le b\le m$. 
Then $\gamma(\cZ,m,n)=\max(am, bn)$. 
\end{prop}

Next we consider rectangular zero-sets $R_{a,b}$, whereby a point
with at least $a$ occupied row neighbors \emph{or} at least $b$ occupied columns neighbors becomes occupied. In this case,
$\gamma$ is given by the following optimization procedure. 

\begin{theorem}\label{thm-line} Assume that $\cZ=R_{a,b}$, where $1\le a\le n$, $1\le b\le m$. 
Then
\begin{equation}\label{eq-line-min}
    \gamma(\cZ,m,n)=\min_{\substack{b\le x\le m\\ a\le y\le n}} \left( (m-x)(n-y)+\max(ax,by)\right), 
\end{equation}
which can be computed in time $\cO(\min(a,b)/\gcd(a,b))$.
In particular, for $a=b$ and $m=n$,  
\begin{equation}\label{eq-line-sym}
\gamma(\cZ,n,n)=
\begin{cases}
    an-a^2/4\,, &\text{if $a\le 2n/3$ and $a$ is even,}\\ 
    an-(a^2-1)/4\,, &\text{if $a\le (2n+1)/3$ and $a$ is odd,}\\
    a^2+(n-a)^2,& \text{otherwise}.
\end{cases}
\end{equation}
\end{theorem}

Our most substantial exact result gives a complete solution to Young domination for triangular Young diagrams $T_a$ from
(\ref{eq-bootstrap}) with $m=n$. 
This case was considered in \cite{BLL}, where various partial results were given. 
In particular, the formula (\ref{eq-boot-even-a}) for even $a$ confirms Conjecture~1 in that paper. 
Note that for $a\ge 2n-1$, the only Young dominating set is $[0,n-1]^2$, so we exclude this case from consideration.  

\begin{theorem}\label{thm-boot} Assume that $n\ge 2$ and let $\cZ=T_a$ be given by (\ref{eq-bootstrap}), where $1\le a\le 2n-2$. If $a$ is even, then 
\begin{equation}\label{eq-boot-even-a}
    \gamma(\cZ,n,n)=\frac{an}2\,.
\end{equation}
If $a$ is odd and $a\le n$, then 
\begin{equation}\label{eq-boot-odd-a<n}
    \gamma(\cZ,n,n) = \frac{a+1}{2}\,n-\frac{a-1}{2}\,.
\end{equation}
Finally, if $a$ is odd and $a>n$, then 
\begin{equation}\label{eq-boot-odd-a>n-simpl}
\gamma(\cZ,n,n)=
\frac{a-1}{2}\, n+
\left\lceil\frac{n(2n-a+1)}{2(2n-a)}
\right\rceil.
\end{equation}
\end{theorem}

For general zero-sets, the Young domination number for $L = 1$ is connected to certain bipartite Turán numbers; see Section~\ref{sec-extremal}.
However, the complexity of computing the Young domination number is open for any $L\ge 1$; see open problems \ref{open-problem-1} and \ref{open-problem-2} in Section~\ref{sec-open} for further discussion.
We provide polynomial-time approximation algorithms, the most precise of which is for latency $L = 1$.

\begin{theorem}\label{thm-3-approx}
There exists an algorithm polynomial in $n$ that takes as input $m,n\in \bN$ with $m\le n$ and a zero-set $\Z$ with $\Z\subseteq R_{n,m}$, and returns a number in the interval $[\gamma, 3\gamma]$ where $\gamma = \gamma(\Z,m,n)$.  
\end{theorem}

 For an arbitrary fixed finite latency number, we give a constant factor approximation algorithm. In its statement, we allow $\cZ$ 
 to be included in a rectangle much smaller than the universe 
 $R_{n,m}$. 

 	\begin{theorem}\label{thm-const-approx} For any fixed finite $L$,
    there exist a constant $C\ge 1$ and an algorithm  that takes as input $m,n\in \bN$ and a zero-set $\Z$ with $\Z\subseteq R_{a,b}\subseteq R_{n,m}$, returns a number in the interval $[\gamma^L, C\gamma^L]$, where $\gamma^L = \gamma^L(\Z,m,n)$, and runs in time that is polynomial in $ab$.  
\end{theorem}

In fact, the algorithms given by Theorems~\ref{thm-3-approx} and~\ref{thm-const-approx} can be easily adapted so that they return a $\cZ$-dominating set $A$ with $|A|\le 3\gamma$ and a $\cZ$-dominating set $A$ with latency $L$ with $|A|\le C\gamma^L$, respectively.
These results place the optimization problems of computing the Young domination numbers $\gamma^L(\Z,m,n)$ in the complexity class \textsf{APX} (see, e.g.,~\cite{ACGKMP}).

Observe that $\gamma^{L+1}(\cZ,m,n)\le \gamma^{L}(\cZ,m,n)$ for all $L\ge 0$.
In light of this and Theorem~\ref{thm-const-approx}, it is natural to ask how much can $\gamma^L$ differ for different $L$. 
The next result addresses this to some extent, showing in particular that at least the first four are not within a constant factor of each other.

\begin{theorem}\label{thm-different-L} Assume that $\cZ\subseteq
R_{a,b}$ for some $a,b\ge 1$.
If $2\le L\le\infty$ and $m,n>ab$, then $\gamma^L(\cZ,m,n)= \gamma^L(\cZ,\infty,\infty)$. If, 
in addition, $L\ge 2ab+5$, then 
$$\gamma^L(\cZ,m,n)= \gamma^\infty(\cZ,\infty,\infty)\in[|\cZ|/4,|\cZ|].
$$
Moreover, 
\begin{equation}\label{eq-different-L}
    \sup_{\cZ,m,n}\frac{\gamma^{L}(\cZ,m,n)}{\gamma^{L+1}(\cZ,m,n)}=\infty,
\end{equation}
for $L=0,1,2$. 
\end{theorem}

\subsection{Related work}

The dynamics with $\cZ=T_a$ given by (\ref{eq-bootstrap}) is also known as 
\emph{bootstrap percolation} \cite{GHPS}, while a rectangular zero-set $\cZ=R_{a,b}$ induces \emph{line growth} \cite{GSS}, also known as \emph{line percolation} \cite{BBLN}. 
In these and related papers on 
growth processes on graphs, points in the set $\cT^t(A)$ are called 
{\it occupied\/} (coded as 1s) and points outside of it {\it empty\/}
(coded as 0s); a set $A$ which makes every point eventually 
occupied is precisely a Young dominating set with $L=\infty$ 
and is also called
a \emph{spanning} \cite{GSS} or \emph{percolating} \cite{BBLN} set. 
For $m=n$, $K_n\Box K_n$ is often referred to as the 
\emph{rook's graph}. 
The $a$-domination number for this case was considered in \cite{BLL} (see also~\cite[Section 8]{HH}), but domination and its variants on \emph{Hamming graphs}, that is, Cartesian products of complete graphs, remain relatively unexplored.   

Young domination is one way to extend $a$-domination 
on graphs \cite{HV}. Another is vector domination (see, e.g.,~\cite{CMV,HPV}), where every vertex has its own requirement for the number of neighbors in the dominating set. 
In a further generalization in the same fashion as in the present paper, 
Cicalese et al.~\cite{CCGMV} considered vector domination with latency.
In social network theory, designing small initial sets that yield large 
final sets due to vector domination with infinite latency is referred 
to as \emph{target set selection} (see~\cite{Chen}).

Domination with infinite latency, that is, spanning or percolation, with random initial sets has played a crucial role in spatial growth processes, starting with the foundational paper \cite{AL}. 
By far the most common setting, with many deep and surprising results (see, e.g.,~the survey \cite{Mor}), is the Cartesian product of $d$ path graphs on $n$ vertices, and thus with standard nearest neighbor lattice connectivity. 
Associated extremal quantities have also received attention, most of all the minimal size of a spanning set, that is, $\gamma^\infty$ in our notation.  
For this graph, the \hbox{$a$-domination} number with infinite latency equals $n^{d-1}$ for $a=d$ \cite{PS} and $\lceil d(n-1)/2\rceil + 1$ for $a=2$ \cite{BBM}.  For 
other values of $a$, a sharp asymptotic formula for the hypercube ($n=2$) 
as $d\to\infty$ was given in \cite{MN}, and an exact formula for $d=a=3$ in \cite{DNR}. 
Smallest spanning sets have also been studied for trees \cite{DR,Rie, BKN}, and strong products \cite{BH} and direct products \cite{BHH} of graphs. 
Among many papers that study graph bootstrap percolation, 
a dynamics akin to ours 
when viewed as an edge-addition rule in bipartite graphs (see~Section~\ref{sec-extremal}), we mention \cite{HHQ} as it also discusses
a method suited for handling product graphs.
For Hamming
graphs, \cite{BBLN} and \cite{BMT} study $\gamma^\infty$
for line growth and bootstrap percolation, while
\cite{GSS} gives bounds on $\gamma^\infty$ 
for general zero-sets. In this context, sets with 
a small latency number $L$ represent fast spanning. 

\subsection{Overview of our methodology and structure of the paper}

 For our exact results (Theorem~\ref{thm-line} and~\ref{thm-boot}), we use two methods to prove lower bounds in different regimes: one is standard bookkeeping, while the other is a novel approach using an algebraic formulation and appropriate inequalities. The upper bounds result from explicit constructions, which, as (\ref{eq-boot-odd-a>n-simpl}) suggests, may be somewhat involved. As already mentioned, 
Theorem~\ref{thm-boot} completely resolves the $a$-domination 
for rook graphs, significantly extending the results from \cite{BLL}. The first approximation result, Theorem~\ref{thm-3-approx}, is proved by comparison with a simpler growth dynamics, which for latency $1$ yields 
an integer optimization problem, which in turn can be solved by dynamic programming. 
The second approximation result, Theorem~\ref{thm-const-approx}, works for general finite latency, and is proved by appropriately restricting 
the number of choices of the initial configuration for the simplified dynamics (Lemma~\ref{lem:poly-gamma''_k}). 
Theorem~\ref{thm-different-L} mostly follows from results from \cite{GSS, GPS}, but we do add an exact result for $L=2$ (Lemma~\ref{lemma-gamma-2}).

We prove Proposition~\ref{prop-L-shape} and Theorem~\ref{thm-line} in Section~\ref{sec-line}, Theorem~\ref{thm-boot} in Section~\ref{sec-boot}, Theorem~\ref{thm-3-approx} in Section~\ref{sec-3-approx}, 
Theorem~\ref{thm-const-approx} in Section~\ref{sec-C-approx}, and Theorem~\ref{thm-different-L} in Section~\ref{sec-comp}.
In Section~\ref{sec-prelim} we introduce some 
tools and prove a few preliminary lemmas. Section~\ref{sec-extremal}
establishes a relationship between our framework and extremal 
problems in bipartite graphs. This connection is based on duality 
between Young diagrams and provides further motivation for 
study of Young domination.
\section{Preliminaries}\label{sec-prelim}

Let $m,n\in \mathbb{N} = \{1,2,3,\ldots\}$.
The vertex set of our graph $K_m\Box K_n$ is $R_{n,m} = [0,n-1]\times [0, m-1]$. 
Also, we shorten $R_n=R_{n,n}$. Thus, two vertices $(j,i)$ and $(\ell,k)$ are adjacent if and only if they differ in precisely one coordinate.
For each $i\in [0,m-1]$ the \emph{$i$-th row} of $R_{n,m}$ is the set of all vertices of $R_{n,m}$ with second coordinate $i$; similarly, for each $j\in [0, n-1]$ the \emph{$j$-th column} is the set of all vertices of $R_{n,m}$ with first coordinate $j$.
Given a vertex $(j,i)\in R_{n,m}$, its row neighborhood $\cN^r(j,i)$ then consists of all vertices distinct from $(j,i)$ with second coordinate $i$,
and its column neighborhood $\cN^c(j,i)$ of all vertices distinct from $(j,i)$
with first coordinate $j$. 
Note that, in this setting, the first factor $K_m$, whose vertices are represented by the interval $[0,m-1]$, is vertical and the second factor $K_n$, represented by the interval $[0,n-1]$, is horizontal; also, each row neighborhood has $n-1$ vertices and each column neighborhood $m-1$ vertices. 
We will also encounter infinite $m$ and $n$: $R_{n,\infty}=[0, n-1]\times [0,\infty)$, and analogously, $R_{\infty,m}$ and $R_{\infty,\infty}$.
Recall that every zero-set $\cZ$ defines the corresponding transformation $\cT = \cT(\cZ)$ on subsets of $R_{n,m}$; note that we may always replace $\cZ$ by $\cZ\cap R_{n,m}$ without affecting the dynamics.

For a Young diagram $\cY$, we call a point
 $(z_1,z_2)\in\bZ_+^2$  its
 \begin{itemize}
 \item {\it convex corner\/} if $(z_1,z_2)\in \cY$
  and $(z_1+1,z_2), (z_1,z_2+1)\notin \cY$;
  \item {\it concave corner\/} if $(z_1,z_2)\notin \cY$
  and $(z_1-1,z_2), (z_1,z_2-1)\in \cY\cup
  (\bZ^2\setminus\bZ^2_+)$.
 \end{itemize}
 We denote by $\cI_\cY$ the set of concave corners of a Young diagram $\cY$.
 Note that elements of $\cI_\cY$ are, under the natural partial order on $\bZ_+^2$, the minimal points of $\bZ_+^2\setminus \cZ$ (i.e., the minimal neighborhood counts that induce occupation).
 Similarly, the convex corners are maximal points
 of $\cZ$.
The set $\cI_\cY$ is minimal in the following sense:
$\cY=\bigcup_{(a-1,b-1)\in \cI_\cY} R_{a,b}$, and
$\cY=\bigcup_{(a-1,b-1)\in \cI} R_{a,b}$
implies $\cI_\cY\subseteq \cI$.
Figure~\ref{fig:intro} provides an example of a $\cZ$-dominating 
set on $K_4\Box K_5$. (In our figures, we often represent a point $z\in\bZ_+^2$ as a unit square with $z$ in its center.)

\begin{figure}[ht!]
\begin{center}
\definecolor{dgrey}{rgb}{0.5, 0.5, 0.5}
\begin{tikzpicture}
\def\u{0.6}
\def\n{5}
\def\m{4}

  \fill [dgrey] 
      (0,0) rectangle   ({1*\u},{3*\u});
      
      \fill [dgrey] 
      (0,0) rectangle   ({3*\u},{2*\u});
      \fill [dgrey] 
      (0,0) rectangle   ({4*\u},{1*\u});
\foreach \i in {0,...,\n}
    {\draw [black, line width=0.1] (\u*\i,0)--(\u*\i,{\u*\m});}
 \foreach \i in {0,...,\m}   
   { \draw [black, line width=0.1] (0,\u*\i)--({\u*\n}, \u*\i);}
   
   \pgfmathsetmacro\mm{{\m-1}}
   \pgfmathsetmacro\nm{{\n-1}}
 \foreach \i in {0,...,\nm}
   {
       \node[] at ({\u*(\i+0.5)},{-0.5*\u}) {$\i$};
   }
   \foreach \i in {0,...,\mm}
   {
       \node[] at ({-0.5*\u},{\u*(\i+0.5)}) {$\i$};
    }
    
 \node[] at ({0.5*\u}, {3.5*\u}) {$\square$};
\node[] at ({1.5*\u}, {2.5*\u}) {$\square$};
\node[] at ({3.5*\u}, {1.5*\u}) {$\square$};
\node[] at ({4.5*\u}, {0.5*\u}) {$\square$};

\begin{scope}[shift={(5,0)}]
\def\n{4}
\def\m{3}
\def\u{1}
\foreach \i in {0,...,\n}
    {\draw [black, line width=0.1] (\u*\i,0)--(\u*\i,{\u*\m});}
 \foreach \i in {0,...,\m}   
   { \draw [black, line width=0.1] (0,\u*\i)--({\u*\n}, \u*\i);}
\fill [black] (0,0) circle (3pt);
  \fill [black] (\u,0) circle (3pt);
   \fill [black] (4*\u,0) circle (3pt);
 \fill [black] (0,\u) circle (3pt);
  \fill [black] (\u,\u) circle (3pt);
   \fill [black] (2*\u,\u) circle (3pt);
   \fill [black] (3*\u,2*\u) circle (3pt);
   \fill [black] (\u,3*\u) circle (3pt);
  \fill [black] (2*\u,3*\u) circle (3pt);
   \fill [black] (4*\u,3*\u) circle (3pt);
  
  \node [left] at (-0.1*\u,0) {$3$};
  \node [left] at (-0.1*\u,\u) {$3$};
  \node [left] at (-0.1*\u,2*\u) {$1$};
  \node [left] at (-0.1*\u,3*\u) {$3$};
   \node [below] at (0,-0.1*\u) {$2$};
   \node [below] at (\u,-0.1*\u) {$3$};
   \node [below] at (2*\u,-0.1*\u) {$2$};
   \node [below] at (3*\u,-0.1*\u) {$1$};
   \node [below] at (4*\u,-0.1*\u) {$2$};
   
\end{scope}
 
\end{tikzpicture}
\end{center}
\caption{Left: a zero-set $\cZ$; the points of $\cZ$ are the centers of the dark squares. Its 
concave corners (see Section~\ref{sec-prelim}) are $(0,3)$, $(1,2)$, $(3,1)$, and $(4,0)$ and are labeled by $\square$. Right: 
a $\cZ$-dominating set $D$ (this time, with its members indicated by $\bullet$) on the graph $K_4\Box K_5$ with vertex set $R_{5,4}=[0,4]\times [0,3]$, of (minimal) size 10.  
Note that the vertices of the product graph are represented so that the first coordinate is vertical and the second coordinate is horizontal, in line with the definitions of rows and columns, and row and column neighborhoods of a vertex.
Accordingly, the row counts and column counts are given, respectively, along the vertical and horizontal axis. 
For every point in $z\in R_{5,4}\setminus D$ 
the point $(r,c)$ with coordinates given by its row and column counts lies outsize $\cZ$, therefore 
$(r,c)$ is greater than or equal to a concave corner of $\cZ$ in the partial order of~$\bZ_+^2$. 
}
\label{fig:intro}
\end{figure}
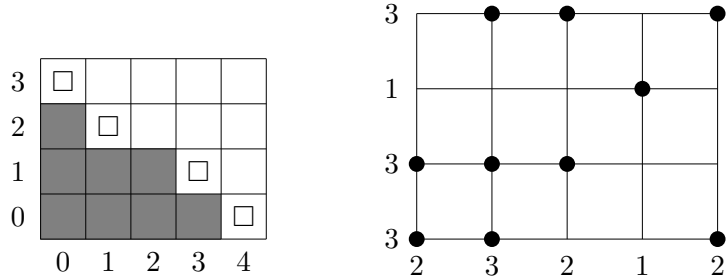

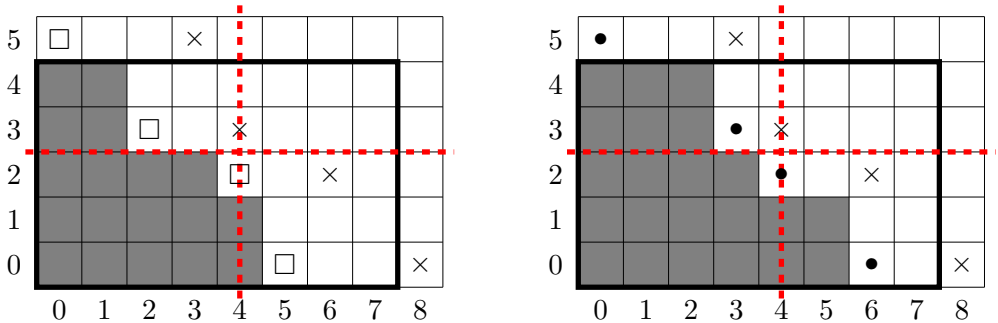
\begin{figure}[ht!]
\begin{center}
\definecolor{dgrey}{rgb}{0.5, 0.5, 0.5}
\begin{tikzpicture}
\def\u{0.6}
\def\n{9}
\def\m{6}

  \fill [dgrey] 
      (0,0) rectangle   ({2*\u},{5*\u});
 \fill [dgrey] 
      (0,0) rectangle   ({5*\u},{3*\u});
 \fill [white] 
      ({4*\u},{2*\u}) rectangle   ({5*\u},{(3*\u});
\foreach \i in {0,...,\n}
    {\draw [black, line width=0.1] (\u*\i,0)--(\u*\i,{\u*\m});}
 \foreach \i in {0,...,\m}   
   { \draw [black, line width=0.1] (0,\u*\i)--({\u*\n}, \u*\i);}
   
   \pgfmathsetmacro\mm{{\m-1}}
   \pgfmathsetmacro\nm{{\n-1}}
 \foreach \i in {0,...,\nm}
   {
       \node[] at ({\u*(\i+0.5)},{-0.5*\u}) {$\i$};
   }
   \foreach \i in {0,...,\mm}
   {
       \node[] at ({-0.5*\u},{\u*(\i+0.5)}) {$\i$};
    }

\draw [black, line width=2] 
      (0,0) rectangle   ({(\n-1)*\u},{(\m-1)*\u});
      
\draw [red, dashed,line width=2] 
      ({-0.25*\u},{3*\u}) --  ({9.25*\u},{3*\u});
\draw [red, dashed,line width=2] 
      ({4.5*\u},{-0.25*\u}) --  ({4.5*\u},{6.25*\u});
             
\node[] at ({2.5*\u}, {3.5*\u}) {$\square$};
\node[] at ({0.5*\u}, {5.5*\u}) {$\square$};
\node[] at ({5.5*\u}, {0.5*\u}) {$\square$};
\node[] at ({4.5*\u}, {2.5*\u}) {$\square$};
\node[] at ({8.5*\u}, {0.5*\u}) {$\times$};
\node[] at ({6.5*\u}, {2.5*\u}) {$\times$};
\node[] at ({3.5*\u}, {5.5*\u}) {$\times$};
\node[] at ({4.5*\u}, {3.5*\u}) {$\times$};

 \fill [dgrey] 
      (12*\u,0) rectangle   ({12*\u+3*\u},{(5*\u});     
 \fill [dgrey]
      (12*\u,0) rectangle   ({12*\u+6*\u},{(2*\u});
\fill [dgrey]
      (12*\u,0) rectangle   ({12*\u+4*\u},{(3*\u});

\foreach \i in {0,...,\n}
    {\draw [black, line width=0.1] (12*\u+\u*\i,0)--(12*\u+\u*\i,{\u*\m});}
 \foreach \i in {0,...,\m}   
   { \draw [black, line width=0.1] (12*\u,\u*\i)--({12*\u+\u*\n}, \u*\i);}
   
   \pgfmathsetmacro\mm{{\m-1}}
   \pgfmathsetmacro\nm{{\n-1}}
 \foreach \i in {0,...,\nm}
   {
       \node[] at ({12*\u+\u*(\i+0.5)},{-0.5*\u}) {$\i$};
   }
   \foreach \i in {0,...,\mm}
   {
       \node[] at ({12*\u-0.5*\u},{\u*(\i+0.5)}) {$\i$};
    }

\draw [black, line width=2] 
      (12*\u,0) rectangle   ({12*\u+(\n-1)*\u},{(\m-1)*\u});
      
\draw [red, dashed,line width=2] 
      ({12*\u-0.25*\u},{3*\u}) --  ({12*\u+9.25*\u},{3*\u});
\draw [red, dashed,line width=2] 
      ({12*\u+4.5*\u},{-0.25*\u}) --  ({12*\u+4.5*\u},{6.25*\u});

\node[] at ({12*\u+0.5*\u}, {5.5*\u}) {$\bullet$};
\node[] at ({12*\u+3.5*\u}, {3.5*\u}) {$\bullet$};
\node[] at ({12*\u+4.5*\u}, {2.5*\u}) {$\bullet$};
\node[] at ({12*\u+6.5*\u}, {0.5*\u}) {$\bullet$};
\node[] at ({12*\u+8.5*\u}, {0.5*\u}) {$\times$};
\node[] at ({12*\u+6.5*\u}, {2.5*\u}) {$\times$};
\node[] at ({12*\u+3.5*\u}, {5.5*\u}) {$\times$};
\node[] at ({12*\u+4.5*\u}, {3.5*\u}) {$\times$};

\end{tikzpicture}
\end{center}
\caption{A Young diagram $\cY$ (dark gray squares) within $R_{8,5}$ (outlined) is given on the left. Its four
concave corners, labeled by $\square$, are reflected through both of the bisecting lines, giving the 
four sites labeled by $\times$. These are reproduced on the right, and determine $\widecheck \cY$,
whose  four concave corners are labeled by $\bullet$. Observe that the double reflection of the 
four $\bullet$'s gives the three points diagonally adjacent to the convex corners of $\cY$, as 
required, because the dual map is an involution. Observe also that
$\widecheck \cY$ equals $R_{8,5}\setminus \cY$, 
reflected through the centroid of $R_{8,5}$, as 
defined in \cite{Roh}.
}
\label{fig:dual}
\end{figure}

We define the following operation on Young diagrams.

\begin{definition}\label{def:dual}
Fix $m,n\in \mathbb{N}$.
    For a Young diagram $\cY\subseteq R_{n,m}$, we define the \emph{dual} of $\cY$ as the unique Young diagram $\widecheck{\cY} \subseteq R_{n,m}$ such that the set of convex corners of $\widecheck{\cY} $ is given by \[\{(n-a-1,m-b-1)\colon (a,b)\in \mathcal{I}_\cY, 0\le a\le n-1, 0\le b\le m-1\}\,.\]
\end{definition}

For a set $A\subseteq R_{n,m}$, let $\rho(A)\subseteq R_{n,m}$ be the reflection of $A$ through the centroid $$((n-1)/2, (m-1)/2)\in\bR_+^2.$$ 
The next lemma, which follows easily from the definition, observes that the dual operation coincides with the one given in, e.g.,~\cite{Roh}. 
See Fig.~\ref{fig:dual} for an example. 

\begin{lemma} For a given $m,n$, and a Young diagram $\cY\subseteq R_{n,m}$, $\widecheck \cY=\rho(R_{n,m}\setminus \cY)$. Therefore, the transformation $\cY\mapsto\widecheck\cY$ is an involution, i.e., it is its own inverse. 
\end{lemma}

We routinely use the following lemma, whose simple proof we omit.

\begin{lemma}\label{lemma-swap}
Let $\cS$ be an operation that, given a subset $A\subseteq R_{n,m}$, returns the subset of $R_{n,m}$ obtained from $A$ by permuting rows $k$ and $k'$, for some distinct $k,k' \in \{1,\dots, m\}$.
Then, for any transformation $\cT$ given by some zero-set, $\cT(\cS(A)) = \cS(\cT(A))$. This also holds for operations permuting columns.
\end{lemma}

We now review a variant of the transformation $\cT$, which adds fixed amounts to each row and column count \cite{GSS, GPS}. The 
advantage is that one can get a nontrivial growth dynamics even if the initial set is $A=\emptyset$, in which case the behavior of iterates is significantly more regular. 
Our transformation $\cT$ is given additional two parameters $\vec r=(r_0, \ldots, r_{m-1})\in \bZ_+^m$
and $\vec c=(c_0, \ldots, c_{n-1})\in \bZ_+^n$, which are sequences of nonnegative integers called \emph{enhancements}. Then 
the \emph{enhanced neighborhood growth} is given by the triple $(\Z,\vec r, \vec c)$  and is defined as follows:
\[
\cT(A)=A\cup\{(j,i)\in R_{n,m}\colon 
    (|\cN^r(j,i)\cap A|+r_i,|\cN^c(j,i)\cap A|+c_j)\notin \Z\}.
\]
We do not introduce a new notation as the usual neighborhood growth given by $\Z$ is the same as its enhancement given by $(\Z,\vec 0,\vec 0)$, and we do not distinguish between the two.
We denote by $|\cdot|$ the $\ell^1$-norm, thus
$|\vec r|=r_0+\cdots+r_{m-1}$ and $|\vec c|=c_0+\cdots+c_{n-1}$.

The next lemma, whose proof is a simple 
verification (see \cite[Lemma 2.3]{GPS}), gives the key regularity property of the enhanced rule.
 
\begin{lemma}\label{lem-zeroyoung}
  Let $m,n \in \bN$  and let $\cT$ be the dynamics defined by the triple $(\Z, \vec r, \vec c)$ consisting of a zero-set $\Z$ and nonincreasing enhancements $\vec r$ and $\vec c$.
  If $A \subseteq R_{n,m}$ is a Young diagram, then so is $\cT(A)$.
  Consequently, $\cT^k(\emptyset)$ is a Young diagram for every $k \in \bZ_+$.
\end{lemma}
 
We conclude this section with two lemmas about the connection between regular and enhanced dynamics. For $m,n \in \bN$, $A\subseteq R_{n,m}$, and $i\in[0,m-1]$, $j\in[0,n-1]$, 
we denote by $\row_i(A)$ and $\col_j(A)$ the number of elements
in the $i$th row and $j$th column of $A$, respectively. 

\begin{lemma}\label{lemma-reg-enh} 
  Let $m,n \in \bN$ and the zero-set $\Z$ be fixed. 
  Assume $A\subseteq R_{n,m}$ and the enhancements $\vr$ and $\vc$ are chosen so that: for each $i\in [0,m-1]$, $\row_i(A)\le r_i$; 
  and for each $j\in [0,n-1]$, $\col_j(A)\le c_j$.
  Let $\cT$ be the regular dynamics (with zero enhancements) and let 
  $\widehat\cT$ be the dynamics defined by $(\Z, \vec r, \vec c)$. 
  Then, for every $k$, 
  $\cT^k(A)\setminus \widehat\cT^k(\emptyset)\subseteq A$.
\end{lemma}

 \begin{proof}
 We prove the equivalent claim that $\cT^k(A)\setminus  A\subseteq  \widehat\cT^k(\emptyset)$ by induction on $k$. For $k=0$, both sets in the claim are empty. 
 Assume now the claim is true for some $k$. 
 Then, $\cT^k(A)\setminus A\subseteq \widehat\cT^k(\emptyset)\subseteq \widehat\cT^{k+1}(\emptyset)$.
 Furthermore, for any point 
 $z=(j,i)\in \cT^{k+1}(A)\setminus\cT^k(A)$, 
 $$|\cN^r(z)\cap \cT^k(A)|=|\cN^r(z)\cap A|+|\cN^r(z)\cap (\cT^k(A)\setminus A)|\le r_i+|\cN^r(z)\cap \widehat\cT^k(\emptyset)|,
 $$
 by the assumption and by the inductive hypothesis. Analogously, 
 $$
 |\cN^c(z)\cap \cT^k(A)|\le c_j+|\cN^c(z)\cap \widehat\cT^k(\emptyset)|.
 $$
 Since $z\in \cT^{k+1}(A)$, the definition of $\cT$ yields that  $(|\cN^r(z)\cap \cT^k(A)|,|\cN^c(z)\cap \cT^k(A)|)\notin \cZ$.
 Hence, also $$ (r_i+|\cN^r(z)\cap \widehat\cT^k(\emptyset)|, c_j+|\cN^c(z)\cap \widehat\cT^k(\emptyset)|)\notin\cZ
 $$ 
 and consequently, $z\in \widehat\cT^{k+1}(\emptyset)$.
 \end{proof}

 Fix $m,n\in \bN$, and enhancement vectors $\vec r\in \bZ_+^m$, $\vec c\in \bZ_+^n$.
For every $A\subseteq R_{n,m}$, we define its \emph{enhancement} $\widehat A$ as an arbitrary but fixed subset of 
$R_{n,m}$ such that: 
$A\subseteq \widehat A$; for every row $i\in [0,m-1]$, 
$\row_i(\widehat A) \ge \min(\row_i(A)+r_i,n)$;
and for every column $j\in [0,n-1]$, 
$\col_j(\widehat A) \ge \min(\col_j(A)+c_j,m)$. 
Thus, we obtain the enhancement by adding $r_i$ elements in every row $i$, up to capacity, and analogously for columns. 

\begin{lemma} \label{lemma-enh-reg} 
Let $\cT$ be the standard dynamics given by $\cZ$, $m$, and $n$, and let $\widehat\cT$ be the  enhanced dynamics given by some $(\Z,\vec r, \vec c)$ with $\vec r\in \bZ_+^m$, $\vec c\in \bZ_+^n$.
Then, each $A\subseteq R_{n,m}$ satisfies $\widehat\cT(A)\subseteq \cT(\widehat A)\,.$
\end{lemma}

\begin{proof}
Since $A\subseteq \widehat A \subseteq \cT(\widehat A)$, it suffices to show that ${\widehat\cT}(A)\setminus A\subseteq \cT(\widehat A)$.
Consider an arbitrary $z=(j,i)\in \widehat\cT(A)\setminus A$. If either 
$\row_i(\widehat A)=n$ or $\col_j(\widehat A)=m$, then $z\in \widehat A\subseteq \cT(\widehat A)$. Otherwise, using the partial order 
on $\bZ_+^2$, 
\begin{equation*}
\begin{aligned}
(|\cN^r(z)\cap\widehat A|, |\cN^c(z)\cap \widehat A|) &=
(\row_i(\widehat A), \col_j(\widehat A))\\
&\ge 
(\row_i(A)+r_i, \col_j(A)+c_j)\\&=(|\cN^r(z)\cap A|+r_i, |\cN^c(z)\cap A|+c_j)\notin\cZ\,, 
\end{aligned}
\end{equation*}  
since $z\in \widehat\cT(A)\setminus A$.
Therefore, $(|\cN^r(z)\cap\widehat A|, |\cN^c(z)\cap \widehat A|)\notin\cZ$, and, hence,
$z\in \cT(\widehat A)$. 
\end{proof}

\section{A connection with extremal graph theory}\label{sec-extremal}

For a family of graphs $\mathcal{F}$, let $\ext(n,\mathcal{F})$ denote the maximum number of edges in a graph on $n$ vertices that does not contain any member of $\mathcal{F}$ as a subgraph. 
The numbers $\ext(n,\mathcal{F})$ are called \emph{Turán numbers}. Their determination is a classical problem in extremal graph theory, even in the special case when $\mathcal{F}$ consists of a single graph $H$ (see, e.g.,~\cite{Bol,AKSV}).

Similarly, the \emph{bipartite Turán problem} considers two positive integers $m$ and $n$ and a family of graphs $\cF$ such that each $H\in \cF$ is a bipartite graph equipped with a bipartition $(X_H,Y_H)$. 
The quantity of interest is then the \emph{bipartite Turán number} $\ext(m,n,\cF)$, the maximum number of edges in a bipartite graph $G$ with bipartition $(X,Y)$ with $|X| = m$ and $|Y| = n$ that does not contain any $H\in \cF$ as a subgraph that respects the given bipartitions (see, e.g.,~\cite{FS}, where the bipartite Turán number is denoted by $\mathbf{ex^\ast}(m,n,\cF)$).
Most research in this direction has addressed the case when $\cF$ contains a single graph $H$. When $H$ is a complete bipartite graph $K_{s,t}$, the problem is known as the 
\emph{Zarankiewicz problem} (see~\cite{Zar}).
While the determination of Zarankiewicz numbers $\ext(m,n,K_{s,t})$ 
is still an open question (see, e.g.,~\cite{CHM,CHM2,Con,ST} for some recent developments), the bipartite Turán problem was completely solved for paths (see~\cite{GyRS}) and partially for cycles (see, e.g.,~\cite{LN}).
A variant of the problem was considered recently by Caro, Patkós, and Tuza (see~\cite{CPT}).

Our results lead to a resolution of the bipartite Turán problem for some families of double stars, including those containing a single double star.
Given two integers $p,q\ge 0$, a \emph{double star} $S_{p,q}$ is the graph obtained from the disjoint union of complete bipartite graphs $K_{1,p}$ and $K_{1,q}$ by adding an edge between a vertex $u$ of degree $p$ in $K_{1,p}$ and a vertex $v$ of degree $q$ in $K_{1,q}$ (note that the vertex $u$ is uniquely determined if $p\neq 1$ and similarly for $v$).
We assume that $S_{p,q}$ is equipped with a bipartition $(X_{p,q},Y_{p,q})$ such that $u\in X_{p,q}$ and $v\in Y_{p,q}$. 
A family $\mathcal{F}$ of double stars is \emph{minimal} if, for all distinct $(p,q)$, $(p',q')$ such that $S_{p,q}, S_{p',q'}\in \cF$, it is not true that $(p,q)\le (p',q')$ (using the partial order in $\bZ_+^2$). 
We will show that the problems of computing the bipartite Turán number for a minimal family of double stars and the Young domination number  are dual to each other, in the sense specified below. 
 
 Assume that $m$ and $n$ are positive integers, $X$ and $Y$ disjoint 
 sets with $|X|=m$ and $|Y|=n$, and that $\mathcal{F}$ is a family of double stars, with fixed bipartitions.  
Observe that, for a given $S_{p,q}\in \cF$, a given bipartite graph $G$ with parts $X$ and $Y$ contains $S_{p,q}$ as a subgraph that respects the bipartitions if and only if there exist vertices $x\in X$, $y\in Y$, so that $xy$ is an edge of $G$, and the degrees of $x$ and $y$ in $G$ are at least $p+1$ and $q+1$, respectively.
Recall that we denote by $\ext(m,n,\mathcal{F})$ the corresponding bipartite Turán number, the maximum number of edges in a bipartite graph $G$ with bipartition $(X,Y)$ that does not contain any $S_{p,q}\in \cF$ as a subgraph that respects the bipartitions.
Note that we may replace $\cF$ by an equivalent minimal 
family by keeping only those $S_{p,q}$ for which $(p,q)$ is minimal.

Let us now connect these notions to Young domination in Hamming rectangles. Recall that, if $\cZ$ is a zero-set, a point $(x,y)$ belongs to $\bZ_+^2\setminus \cZ$ if and only if $(a,b)\le (x,y)$ for some concave corner $(a,b)$ of $\Z$.
To any minimal family $\cF$ of double stars, we associate a Young diagram $\cY_\cF$ whose concave corners are all points $(q,p)$ such that $S_{p,q}\in\cF$.

\begin{theorem}\label{thm-zero-sets-double-stars}
Fix $m,n\in\bN$.
Assume that a zero-set $\Z$  and a minimal family 
$\cF$  of double stars are such that 
$\cZ$ and $\cY_\cF$ are both subsets of $R_{n,m}$
and duals of each other, in the sense of Definition~\ref{def:dual}. 
Then, 
\[\gamma(\cZ, m,n)+\ext(m,n,\mathcal{F}) = mn \,.\]
\end{theorem}

\begin{proof}
By definition, $\gamma(\cZ, m, n)$ is the minimum cardinality of a set $A\subseteq V(K_m\Box K_n)$ such that every vertex $v\in V(K_m\Box K_n)\setminus A$ satisfies $(|\cN^{r}(v)\cap A|,|\cN^{c}(v)\cap A|)\not\in\cZ$, that is, there exists a concave corner $z = (a,b)$ of $\cZ$ such that $|\cN^{r}(v)\cap A|\ge a$ and $|\cN^{c}(v)\cap A|\ge b$.
Let us identify the vertices of $K_m\Box K_n$ with the edges of the complete bipartite graph $K_{m,n}$ with bipartition $(X,Y) = (V(K_m),V(K_n))$.
Then, $A$ corresponds to a set of edges of $K_{m,n}$ such that for every edge $(i,j)$ of $K_{m,n}$ that is not in $A$ there exists a concave corner $(a,b)\in \mathcal{I}_\Z$ such that vertex $i\in X$ is incident with at least $a$ edges in $A$ and vertex $j\in Y$ is incident with at least $b$ edges in $A$.
Let $A^c$ denote the bipartite complement of $A$ in $K_{m,n}$, that is, the bipartite graph with bipartition $(X,Y)$ in which two vertices $i\in X$ and $j\in Y$ are adjacent if and only if $(i,j)\not\in A$.
The above condition on $A$ can be equivalently phrased in terms of $A^c$, as follows: for every edge $(i,j)$ of $A^c$ there exists a concave corner $(a,b)\in \mathcal{I}_\Z$ such that $i$ is incident with at most $n-a$ edges in $A^c$ and $j$ is incident with at most $m-b$ edges in $A^c$.
This is in turn equivalent to the following condition: for all $p,q\in \mathbb{Z}_+$, if a double star $S_{p,q}$ is a subgraph of $A^c$ such that the vertex of degree $p+1$ belongs to $X$ and the vertex of degree $q+1$ belongs to $Y$, then there exists a concave corner $(a,b)\in \mathcal{I}_\Z$ such that $p\le n-a-1$ and $q\le m-b-1$ (note that since $p\ge 0$, we have that $a\le n-1$ and, similarly, $b\le m-1$).
In other words, the double stars that can appear as subgraphs of $A^c$ are precisely the double stars $S_{p,q}$ such that $(p,q)$ is an element of the dual $\widecheck\Z$ of $\Z$.
Equivalently, the minimal family of double stars that are forbidden for $A^c$ are precisely the stars in $\mathcal{F}$.
Hence, the maximum possible number of edges in $A^c$ is given by the value of $\ext(m,n,\mathcal{F})$.
Since the number of edges in $A^c$ equals $mn-|A|$, minimizing the cardinality of $A$ is equivalent to maximizing the cardinality of $A^c$, which implies $\gamma(\cZ, m,n)  = mn-\ext(m,n,\mathcal{F})$, as claimed.
\end{proof}

Our results on exact evaluation of $\gamma$ then immediately give the following results. 
In particular, the next corollary is given by Proposition~\ref{prop-L-shape}.
Note that the double star $S_{a,0}$ is isomorphic to the complete bipartite graph $K_{1,a+1}$ with parts $X_{a,0}$ and $Y_{a,0}$ such that $X_{a,0}$ consists of a single vertex of degree $a+1$ (and, hence, every vertex in $Y_{a,0}$ has degree $1$).
A similar observation holds for the double star $S_{0,b}$. 
Therefore, for $\cF=\{S_{a,0}, S_{0,b}\}$, the quantity $\ext(m,n,\mathcal{F})$ counts the maximum number of edges in a bipartite graph with parts $X$ and $Y$ of sizes $m$ and $n$, respectively, such that each vertex in $X$ has degree at most $a$ and each vertex in $Y$ has degree at most $b$.

\begin{corollary}\label{cor-V-domination}
For $0\le a\le n$, $0\le b\le m$, we have $\widecheck R_{a,b}=V_{n-a, m-b}$. 
Consequently, for the 
minimal family $\cF=\{S_{a,0}, S_{0,b}\}$, 
$$\ext(m,n,\mathcal{F})=\min\{am,bn\}\,.$$ 
\end{corollary}

\begin{proof} Note that $\cY_\cF=R_{a,b}$.
The equality $\widecheck R_{a,b}=V_{n-a, m-b}$ follows directly from the definitions.
By Proposition~\ref{prop-L-shape}, 
$\ext(m,n,\mathcal{F})=mn-
\gamma(V_{n-a, m-b},m,n) = mn-\max((n-a)m, (m-b)n) = 
\min\{am,bn\}$.
\end{proof}

\begin{corollary}\label{cor-L-domination}
For $0\le a\le n$, $0\le b\le m$, we have 
$\widecheck V_{a,b}=R_{n-a, m-b}$. Consequently, for the 
minimal family $\cF=\{S_{a,b}\}$, 
$$\ext(m,n,\mathcal{F})=mn-\gamma(R_{n-a,m-b}, m,n), 
$$  
given by Theorem~\ref{thm-line}.
\end{corollary}
\begin{corollary}\label{cor-boot-domination} For $1\le a\le m+n-2$, 
the dual of $T_a\cap R_{n,m}$, in the sense of Definition~\ref{def:dual}, is $T_{m+n-1-a}\cap R_{n,m}$. 
Consequently, for $m=n$, $a\le 2n-2$, and the minimal family
 $\cF=\{S_{i,j}\colon i+j=a\}$, we have
$$\ext(n,n,\mathcal{F})=n^2-\gamma(T_{2n-1-a}, n,n), 
$$  
given by Theorem~\ref{thm-boot}.
\end{corollary}

\section{Young domination with $L$-shaped and rectangular zero-sets}\label{sec-line}

\begin{lemma}\label{lemma-ab-fill} Fix integers
$n\ge a\ge 1$ and $m\ge 1$, and let $R_{n,m}$ be the vertex set of the graph $K_m\Box K_n$.
Then, there exists a set $S\subseteq R_{n,m}$ such that each row of $R_{n,m}$ 
contains exactly $a$ elements 
of $S$, and every column contains at least $\lfloor am/n \rfloor$ and at 
most $\lceil am/n \rceil$ elements of $S$. 
\end{lemma}

\begin{proof}
On each row, $S$ contains an interval of $a$ sites, starting at the horizontal position of the first site after the interval of the previous row ends, 
with wraparound boundary. 
More formally, on each row $i\in [0,m-1]$ the interval is at column indices $ia, ia+1, \ldots, ia+a-1$, 
taken modulo $n$. 
Every time one of these column indices 
reaches $n-1$, the minimal number of elements in $S$
in every column increases by 1, 
and every time one of these column indices 
reaches $0$, the maximal number of elements in every column increases by 1. 
Therefore, the final minimal (resp.~maximal) number of sites in every column 
is 
$\lfloor am/n \rfloor$ (resp.~$\lceil am/n \rceil$). 
\end{proof}

We can now prove Proposition~\ref{prop-L-shape}, which determines the minimum cardinality $\gamma$ of a \hbox{$\cZ$-dominating} set where the requirement is that each unoccupied vertex has at least $a$ occupied row neighbors and at least $b$ occupied column neighbors.

\begin{proof}[Proof of Proposition~\ref{prop-L-shape}]
We may assume that $am\ge bn$ and $a\ge 1$. 
It is clear that every row needs at least 
$a$ occupied sites, so $\gamma\ge am$. Now we need to construct a configuration in which every row contains exactly $a$ sites, 
and also every column contains $b$ or more sites. As  
$\lfloor ma/n\rfloor\ge b$, the above lemma provides this. 
\end{proof}

For the rest of this section we assume that $\cZ=R_{a,b}$, where $1\le a\le n$, $1\le b\le m$, and prove Theorem~\ref{thm-line}. We recall that, in this case, 
a $\cZ$-dominating set $S$ is such that for every $z\in R_{n,m}\setminus S$ there are at least $a$ points of $S$
in the row of $z$ or at least $b$ points of $S$ in the column of $z$. 

\begin{proof}[Proof of (\ref{eq-line-min})]
In this proof, we count rows starting from the top and 
columns starting from the left, as for matrices. 
Also, the dimensions of a rectangular block within a matrix are given as $(\text{number of rows})\times (\text{number of columns})$.

To prove the lower bound, assume $S$ is a set that realizes $\gamma$. 
For some $x\in[0,m]$, exactly $m-x$ rows 
contain at most $a-1$ elements of $S$, 
and we may assume those are the first $m-x$ rows. For some $y\in[0,n]$, exactly $n-y$ columns do not lack any of the elements of $S$ among the first $m-x$ row positions, and we may assume those are the leftmost columns. 
It follows that $S$ contains the $(m-x)\times (n-y)$ block in the top left corner, that each of 
the rows $m-x+1,\ldots, m$ contains at least $a$ elements of $S$ (by definition of $x$), and all 
the columns $n-y+1,\ldots, n$ contain at least $b$ elements of $S$ (as $S$ is $\cZ$-dominating). 
This gives the lower bound of the same form as 
(\ref{eq-line-min}), except the lower bounds on $x$ and $y$ are $1$ instead of $b$ and $a$. 
We next show that the minimum is achieved 
for $x\ge b$ and $y\ge a$. 

If $x<b$, then we may, without changing the size of $S$ or affecting its Young dominance, move points of $S$
within the last $y$ columns downwards until $S$ contains the entire $x\times y$ block. Analogously, we may also assume that $S$ contains 
this block when $y<a$. 

Assume first that $x<b$ and $y\ge a$. Then, in the last $y$ rows, $S$ has exactly $b-x$ points outside the $x\times y$ block in each column. 
Therefore
\begin{equation*}
\begin{aligned}
|S|&=(m-x)(n-y)+ax+y(b-x)
\\&=(a-n)x+(b-m)y+mn, 
\end{aligned}
\end{equation*}
which is nonincreasing in $x$, 
and so we can construct a new $\cZ$-dominating set $S'$, with $x$ replaced by $x+1$, such that $|S'|\le |S|$. 
The case $x\ge b$, $y<a$ is   eliminated similarly. 

If $x<b$, $y<a$, then $|S|$ has a similar form:
\begin{equation*}
\begin{aligned}
|S|&=(m-x)(n-y)+xy+y(b-x)+x(a-y)\\
&=
(a-n)x+(b-m)y +mn,
\end{aligned}
\end{equation*}
so the size of $S$ decreases if we increase either $x$ 
or $y$. This establishes the $\ge$ part of 
(\ref{eq-line-min}). 

To prove the $\le$ part, we construct 
a $\cZ$-dominating set of size 
$(m-x)(n-y)+\max(ax,by)$ whenever
$b\le x\le m$ and $a\le y\le n$. 
We may assume that $ax\ge by$. 
Our set will have the block structure
$$
B=\begin{bmatrix}
 &B_{11}^{(m-x)\times (n-y)}
 &B_{12}^{(m-x)\times y}\\
 &B_{21}^{x\times (n-y)}
 &B_{22}^{x\times y}
\end{bmatrix},
$$
with the blocks $B_{12}$ and $B_{21}$ comprised of $0$s, 
the block $B_{11}$ comprised of $1$s, and the block 
$B_{22}$ chosen so that there are exactly $a$ 1s in each 
row and at least $b$ 1s in each column. This is possible by 
Lemma~\ref{lemma-ab-fill} and clearly the resulting 
$B$ gives a $\cZ$-dominating set.
\end{proof}

\begin{proof}[Proof of (\ref{eq-line-sym})]
   Let $\varphi(y)=(n-y)^2+ay$ for all $y\in \mathbb{Z}$.
   We may assume that $x\le y$ in (\ref{eq-line-min}), which, by minimizing 
    over $x$ first, 
    gives $\gamma=\min_{a\le y\le n} \varphi(y)$. 
    But $\varphi$ achieves 
    its global minimum at $\lceil (2n-a-1)/2\rceil$.
    Therefore, $\gamma = \varphi(\max(a,\lceil (2n-a-1)/2\rceil))$, which can be checked to equal the expression in (\ref{eq-line-sym}).
\end{proof}

\begin{proof}[Proof of the upper bound on the computation time of (\ref{eq-line-min})] 
We may assume that $a\le b$. 
For a fixed $y$, optimization over $x$ in the two cases (when the maximum 
equals its first or its second argument) amounts to minimizing a linear 
function, and thus is achieved at $x=\lfloor by/a\rfloor$ or $x=\lceil by/a\rceil$ or at an expression that does not depend on $y$. 
Let $a_1=a/\gcd(a,b)$ and $b_1=b/\gcd(a,b)$. 
If we write $y=ja_1+r$, where $0\le r< a_1$, then $\lfloor by/a\rfloor=b_1\cdot j+\lfloor b_1r/a_1\rfloor$
and $\lceil by/a\rceil=b_1\cdot j+\lceil b_1r/a_1\rceil$. 
We then perform optimization over $j$ first, for a fixed $r$. In every case, we need to compute 
the minimum of a function that is at most quadratic in $j$, which we 
can solve explicitly. Finally, we perform minimum over $r$, which we can do in $\cO(a_1)$ steps.
\end{proof}

\section{Exact results for $a$-domination on Hamming squares}\label{sec-boot}

Throughout this section, we consider $a$-domination on the 
Hamming square, that is, we assume $m=n$, that $\cZ$ is the 
triangular set $T_a$ given by (\ref{eq-bootstrap}), and study 
$\gamma=\gamma(\cZ,n,n)$. We prove Theorem~\ref{thm-boot} in the 
three cases in order in the next three subsections. 

\subsection{Even $a$}\label{subsec-even-a}

\begin{lemma}\label{lemma-boot-alla-lb}
For all positive integers $a$ and $n$ such that $a\le 2n$, it holds that $\gamma\ge an/2$.
\end{lemma}

\begin{proof}
To prove that $\gamma\ge an/2$, we begin with 
a reformulation of the claim:
if $x_{ij}\in\{0,1\}$, $i,j=1,\ldots, n$, are such that for every $i,j$
\begin{equation}\label{eq-even-a-1}
\sum_{k, k\ne j}(1-x_{ij})x_{ik}+\sum_{k, k\ne i}(1-x_{ij})x_{kj}\ge a(1-x_{ij}), 
\end{equation}
then 
\begin{equation}\label{eq-even-a-2}
\sigma:=\sum_{i,j}x_{ij}\ge \frac{na}2.
\end{equation}

Let $r_i=\sum_j x_{ij}$ and $c_j=\sum_i x_{ij}$ 
be the row and column sums. 
Summing (\ref{eq-even-a-1}) over all $i,j$, the right-hand side sums to $a(n^2-\sigma)$. The first terms of the left-hand side sum to 
\begin{equation}\label{eq-even-a-3}
\begin{aligned}
\sum_i \sum_{j,k,k\ne j} (1-x_{ij})x_{ik}&=\sum_i|\{(k,j)\colon x_{ik} = 1\text{ and }x_{ij} = 0\}|\\
&=\sum_i r_i(n-r_i)\\&=n\sigma-\sum_i r_i^2.
\end{aligned}
\end{equation}
Similarly, the second terms on the left-hand side of (\ref{eq-even-a-1}) sum to 
\begin{equation}\label{eq-even-a-4}
n\sigma-\sum_j c_j^2.
\end{equation}
Note that, by the Cauchy–Bunyakovsky–Schwarz inequality applied to the vectors \hbox{$(1,\ldots, 1)\in \mathbb{R}^n$} and $(r_1,\ldots, r_n)$, we have that $\sigma^2 = (\sum_i r_i)^2\le n\cdot \sum_i r_i^2$.
Therefore
$$
\sum_i r_i^2\ge \sigma^2/n, \quad \sum_j c_j^2\ge \sigma^2/n, 
$$
so that (\ref{eq-even-a-3}, \ref{eq-even-a-4}) yield the inequality
$$
2n\sigma-\frac{2\sigma^2}{n}\ge an^2-a\sigma,
$$
that is,
\begin{equation}\label{eq-even-a-5}
2\sigma^2-n(2n+a)\sigma+an^3\le 0, 
\end{equation}
and, as $a\le 2n$, the left zero of the quadratic function is at $an/2$. 
Thus, (\ref{eq-even-a-2}) follows.
\end{proof}

\begin{proof}[Proof of Theorem~\ref{thm-boot} for even $a$]
Here, we need to prove (\ref{eq-boot-even-a}). 
The inequality $\gamma\le an/2$ is proved in Corollary 11 of \cite{BLL}, but also follows from 
observing that the set in Lemma~\ref{lemma-ab-fill}, with $m=n$ and $a$ replaced by $a/2$, is $a$-dominating. The opposite inequality follows 
from Lemma~\ref{lemma-boot-alla-lb}. 
\end{proof}

\subsection{Odd $a\le n$}\label{subsec-odd-a<n}

We first state a lemma on a common design of dominating sets. 

\begin{lemma}\label{lemma-domination-check}
Assume that a set $D\subseteq R_{n}$ contains at least $b$ elements in each row and column of $R_{n}$. 
Furthermore, assume that $(i,j)\in D$ for every 
$(i,j)$ such that both row $i$ and column $j$ contain exactly $b$ elements
of $D$. Then $D$ is $(2b+1)$-dominating. 
\end{lemma}

\begin{proof}
By permuting rows and columns, we may assume that, for some $(k,\ell)$, $[0,k-1]\times [0,\ell-1]\subseteq D$, that the first $k$ rows and the first $\ell$ columns each contain exactly $b$ elements of $D$, and that the remaining rows and columns contain at least $b+1$ elements of $D$ each.

To see that $D$ is a $(2b+1)$-dominating set in $R_{n}$, consider an arbitrary vertex $(i,j)\in R_{n}\setminus D$.
If $i\le k-1$, then $j\ge \ell$, and vertex $(i,j)$ has $b$ row neighbors in $D$ and $b+1$ column neighbors in $D$; hence, it has $2b+1$ neighbors in $D$.
The argument is similar if $j\le \ell-1$.
Suppose now that $i\ge k$ and $j\ge \ell$.
Then, $(i,j)$ has $b+1$ row neighbors in $D$, as well as $b+1$ column neighbors in $D$, for a total of $2b+2$ neighbors in $D$.
\end{proof}

\begin{lemma}\label{lemma-odd-a-upper-bound}
If $a=2b+1$, for some integer $b\in [0, \lfloor (n-1)/2\rfloor]$,
then it holds that 
\[\gamma\le (b+1)n-b\,.\]
\end{lemma}

\begin{proof}
We construct a $(2b+1)$-dominating set $D$ of $K_n\Box K_n$ with cardinality $(b+1)n-b$ as follows.
Let $B = [0,b-1]$ and $\overline{B} = [0,n-1]\setminus B$.
Let $D = (B\times B)\cup C$, where $C$ is an arbitrary subset of the set $\overline{B} \times \overline{B} $ such that each row and column of $\overline{B} \times \overline{B} $ contains precisely $b+1$ elements of $C$.
Note that such a set exists by Lemma~\ref{lemma-ab-fill}, 
since $b+1\le n-b = |\overline{B}|$.
By Lemma~\ref{lemma-domination-check}, $D$ is a $(2b+1)$-dominating set.
Since $|D| = b^2+(b+1)(n-b) = (b+1)n-b$, we obtain that $\gamma\le (b+1)n-b$, as required.
\end{proof}

We next provide a matching lower bound to the upper bound given by Lemma~\ref{lemma-odd-a-upper-bound}.

\begin{lemma}\label{lemma-odd-a-lower-bound}
If $a=2b+1$, for some integer 
$b\in [0, \lfloor (n-1)/2\rfloor]$,
then it holds that
\[\gamma \ge (b+1)n-b\,.\]
\end{lemma}

\begin{proof}
Let $D$ be a $(2b+1)$-dominating set of $K_n\Box K_n$.
We show that $|D|\ge (b+1)n-b$.

Suppose for a contradiction that $|D| < (b+1)n-b$.
For $i\in [0,n-1]$, let $d_i$ be the number of vertices in $D$ in row $i$ (i.e., with first coordinate $i$), and let $\ell = \min_{1\le i\le n} d_i$.

We prove the lemma via a sequence of claims established in steps \eqref{l<=b}--\eqref{b+1-column} below.

\sta{$\ell \le b$. \label{l<=b}}
If $\ell\ge b+1$, then $|D|\ge (b+1)n$, a contradiction.
This proves \eqref{l<=b}.

\sta{$\ell = b$. \label{l=b}}
Call a subset of $R_n$ \emph{$k$-thick} if every row and every column contains at least $k$ points of the set. 
Suppose we know, for some integer $k\ge 0$, that $D$ is $k$-thick. 
Suppose also that 
$k^2+(n-k)(a-k)\ge bn+n-b$. Then we claim that $D$ is also $(k+1)$-thick. Assume not. Then we may without loss of generality assume that 
the first row (i.e., row 0) contains exactly $k$ points of $D$ in the leftmost $k$ positions. 
Since $D$ is $k$-thick, each of the leftmost $k$ columns contains at least $k$ points of $D$. As $D$ is $a$-dominating, each of the other 
$n-k$ columns contains at least $a-k$ points of $D$. Therefore, 
$|D|\ge k^2+(n-k)(a-k)\ge bn+n-b$, which contradicts our assumption on the size 
of $D$. 

It remains to show that, 
for $2b+1\le n$ and $0\le k\le b-1$,
$$
 k^2+(n-k)(2b+1-k)\ge bn+n-b.
$$
In fact the above inequality holds for $k\in[0,b]$. 
To see this, observe that the derivative of the left-hand side with respect to $k$ is $4k-(n+2b+1)\le 2b-1-n<0$, and that at $k=b$ the two sides are equal.
This shows that, since $D$ is trivially $0$-thick, $D$ is also $b$-thick, concluding the proof of \eqref{l=b}.

Let us call a row index $i\in [0,n-1]$ \emph{light} if $D$ contains exactly $b$ vertices from the $i$-th row.
Let $I$ be the set of all light row indices.

\sta{$|I|\ge b+1$.\label{b+1-light-rows}}
Suppose for a contradiction that $|I|\le b$.
Summing up the elements of $D$ row by row, we obtain $|D| \ge |I|\cdot b + (n-|I|)\cdot (b+1)\ge (b+1)n-b$, a contradiction.
This proves \eqref{b+1-light-rows}.

By \eqref{b+1-light-rows} and by permuting the rows if necessary, we may assume without loss of generality that $[0,b]\subseteq I$.

\sta{$D$ contains at least $b+1$ vertices from each column.\label{b+1-column}}
Consider a column index $j\in [0,n-1]$. 
If its first $b+1$ vertices belong to $D$, that is, $(i,j)\in D$ for all $i\in [0,b]$, then we are done.
So we may assume that this is not the case, that is, there exists a row index $i\in [0,b]$ such that $(i,j)\not\in D$.
Note that $i\in I$, since $[0,b]\subseteq I$.
Hence, the vertex $(i,j)$ has exactly $b$ row neighbors in $D$.
Since $D$ is a $(2b+1)$-dominating set in $K_n\Box B_n$, we infer that the vertex $(i,j)$ has at least $b+1$ column neighbors in $D$.
This proves \eqref{b+1-column}.

Summing up the elements of $D$ column by column, we obtain using \eqref{b+1-column} that $|D| \ge (b+1)\cdot n$, a contradiction.
\end{proof}

\begin{proof}[Proof of Theorem~\ref{thm-boot} for odd $a\le n$]
This follows immediately from Lemmas~\ref{lemma-odd-a-upper-bound} 
and~\ref{lemma-odd-a-lower-bound}. 
\end{proof}

\subsection{Odd $a>n$}\label{subsec-odd-a>n}
Observe that  (\ref{eq-boot-odd-a>n-simpl}) can be rewritten as
\begin{equation}\label{eq-boot-odd-a>n}
\gamma=\frac{a-1}{2}\, n+\left\lceil\frac n2\right\rceil+
\left\lceil\frac{n}{2(2n-a)}-
\frac 12\cdot 
\mathbbm 1_{\text{$n$ odd}}\right\rceil,
\end{equation}
and we call the third term the \emph{correction term}.
The next lemma provides a lower bound, which turns out to be sharp in this case (as well as in the case $a = n$), but not for $a<n$ (where the approach in the proof of Lemma~\ref{lemma-odd-a-lower-bound} is superior). 
For the upper bound, 
we use a block construction, which is similar to the one in 
Lemma~\ref{lemma-odd-a-upper-bound}, but a bit more
involved. 

\begin{lemma}\label{lemma-a>n-lower-bound} 
For $s\in \bZ$, let 
$$\nu(s) = -2n\cdot\lfloor s/n\rfloor^2 +2(2s-n)\cdot \lfloor s/n\rfloor+an^2+(2-a-2n).
$$
Then, for any $a\le 2n-2$, 
\begin{equation}\label{eq-a>n-lb}
\begin{aligned}
\gamma \ge  \min\{
&s\in \left[\lceil an/2\rceil, \lceil a/2\rceil n\right]\colon \nu(s)\le 0\}.
\end{aligned}
\end{equation}
\end{lemma}

\begin{proof}
For a fixed integer $n\ge 1$ and integer $s\ge 0$ we define  
$$\mu(s)=\min\left\{r_1^2+\cdots+r_n^2\colon r_1,\ldots, r_n\in \bZ_+, \sum_{i=1}^n r_i=s\right\}.
$$
We first claim that the numbers $r_i$ at which the minimum is achieved 
must either be all equal or else have two consecutive values.
Indeed, if, say, $r_2\ge r_1+2$, then  
$$
(r_1+1)^2+(r_2-1)^2=r_1^2+r_2^2+2(1-r_2+r_1)<r_1^2+r_2^2
$$
and so replacing $r_1$ and $r_2$ with $(r_1+1)$ and $(r_2-1)$, respectively, decreases the sum of squares. 
It follows that, with $r=\lfloor s/n\rfloor$ and 
$t=s-rn$, 
$$
\mu(s)=(n-t)r^2+t(r+1)^2=nr^2+t(2r+1).
$$
    From the proof of Lemma~\ref{lemma-boot-alla-lb}, 
    we get that any feasible solution $(x_{ij})$, with $\sigma=\sum_{ij}x_{ij}$, must satisfy
$$
2\mu(\sigma)\le -an^2+(a+2n)\sigma.
$$
Furthermore, for the optimal solution, $\sigma\ge \lceil an/2\rceil$
(Lemma~\ref{lemma-boot-alla-lb}) and 
$\sigma\le \lceil a/2\rceil n$ (Corollaries 11 and 12 in \cite{BLL}). Simple algebra 
yields (\ref{eq-a>n-lb}).
\end{proof}

\begin{proof}[Proof of Theorem~\ref{thm-boot} when $a>n$, $a$ is odd, and $n$ is even] 
Write $a=2b+1$ and $n=2m$. 
We first design a dominating 
set with its size given by (\ref{eq-boot-odd-a>n}). 
In this proof rows are again counted starting from the top and 
columns starting from the left and the dimensions of a block 
are given as $(\text{number of rows})\times (\text{number of columns})$.
The $a$-dominating set will be constructed as a block matrix with four blocks (where as usual, vertices in the set are encoded by $1$s): 
$$
B=\begin{bmatrix}
 B_{11}^{(m-k)\times (m-k)}
 &B_{12}^{(m-k)\times (m+k)}\\
 B_{21}^{(m+k)\times (m-k)}
 &B_{22}^{(m+k)\times (m+k)}
\end{bmatrix}.
$$
The dimensions depend on the number $k$, which will be specified below.
The matrix $B$ will be symmetric, and will have $b$ 1s in each of its first $m-k$ rows and $b+1$ 1s in each of its remaining $m+k$ rows. 
Therefore, the number of 1s is $2mb+m+k$, so $k$ will be exactly the correction term in (\ref{eq-boot-odd-a>n}). 
All entries of $B_{11}$ are set to $1$. 
Note that $B$ determines an $a$-dominating set by Lemma~\ref{lemma-domination-check}. 
 
We {\it start\/} by making all entries in $B_{22}$ $1$ as well, but this will later change. 
We now concentrate on $B_{12}$, so all rows and columns are of that block only, unless otherwise specified. 
The goal is to define nonnegative integers $c$ and $d$, and design a configuration in $B_{12}$ which has $d$ 1s in every row and either $c$ or $c+1$ 1s in every column; we now decide what $c$ and $d$ must be. 
First, $c$ is the number that, together with the $m+k$ 1s in the same column of $B_{22}$ makes for $b+1$ $1$s, that is, $c=b-m-k+1$. 
(We will choose $k$ so that $k\le b-m+1$; this will ensure that $c$ is nonnegative.)
Furthermore, 
$d$ is the number that, together with the $m-k$ 1s in the same row
of $B_{11}$ makes for $b$ $1$s, that is, $d=b-m+k$. In order 
for this construction to be possible, 
\begin{equation}\label{eq-consistency}
    (m-k)d\ge (m+k)c, 
\end{equation}
which works out to be 
$$
m(4k-1)\ge (2b+1)k=ak, 
$$
and so $k$ needs to be at least the correction term in 
(\ref{eq-boot-odd-a>n}).

We next show that the claimed construction in $B_{12}$ is 
possible if $k$ is the smallest positive integer such that (\ref{eq-consistency}) holds. 
We first claim that we then have 
\begin{equation}\label{eq-consistency-1}
    (m-k)d\le  (m+k)(c+1). 
\end{equation}
If $k=1$, $d=c+1$ and the claim holds. Assume now that $k\ge 2$. 
In this case, we know that (\ref{eq-consistency}) does not hold if we replace 
$k$ by $k-1$, which yields
$$
b>\frac{4k-3}{2(k-1)}m-\frac 12.
$$
Equation (\ref{eq-consistency-1}) is equivalent to 
$$
b\ge \frac{2k-1}{k}m-1,
$$
so it holds if
$$
\frac{4k-5}{2(k-1)}\ge \frac{2k-1}{k},
$$
which holds when $k\ge 2$, establishing (\ref{eq-consistency-1}). 

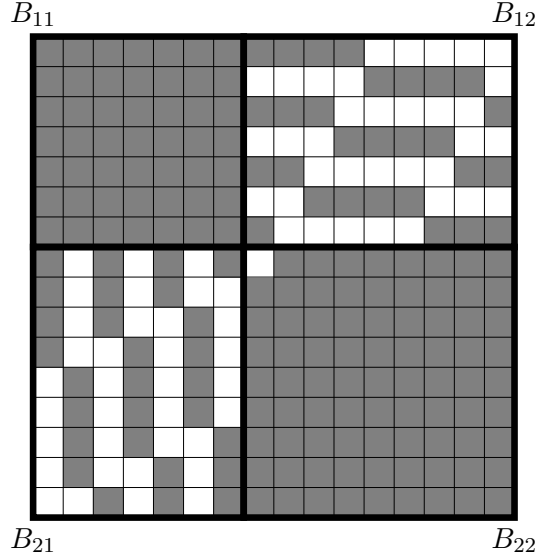
\begin{figure}[ht!]
\begin{center}
\definecolor{dgrey}{rgb}{0.5, 0.5, 0.5}
\begin{tikzpicture}
\def\u{0.4}
\def\m{8}
\def\d{4}
\pgfmathsetmacro\n{{2*\m}}

\fill [dgrey] 
      (0,{(\m+1)*\u}) rectangle   ({(\m-1)*\u},{(\n*\u});
\fill [dgrey] 
       ({(\m-1)*\u},0) rectangle   ({\n*\u},{((\m+1)*\u});
\fill [white] 
       ({(\m-1)*\u},{\m*\u}) rectangle  ({(\m-1)*\u+\u},{\m*\u+\u});

\foreach \i in {0,1,2,3}
\fill [dgrey] 
       ({(\m-1)*\u+\i*\u},{(\n-1)*\u}) rectangle  ({(\m-1)*\u+\i*\u+\u},{(\n-1)*\u+\u});

\foreach \i in {4,5,6,7}
\fill [dgrey] 
       ({(\m-1)*\u+\i*\u},{(\n-1)*\u-\u}) rectangle  ({(\m-1)*\u+\i*\u+\u},{(\n-1)*\u-\u+\u});
       
\foreach \i in {8,0,1,2}
\fill [dgrey] 
       ({(\m-1)*\u+\i*\u},{(\n-1)*\u-2*\u}) rectangle  ({(\m-1)*\u+\i*\u+\u},{(\n-1)*\u-2*\u+\u});

\foreach \i in {3,4,5,6}
\fill [dgrey] 
       ({(\m-1)*\u+\i*\u},{(\n-1)*\u-3*\u}) rectangle  ({(\m-1)*\u+\i*\u+\u},{(\n-1)*\u-3*\u+\u});

\foreach \i in {7,8,0,1}
\fill [dgrey] 
       ({(\m-1)*\u+\i*\u},{(\n-1)*\u-4*\u}) rectangle  ({(\m-1)*\u+\i*\u+\u},{(\n-1)*\u-4*\u+\u});

\foreach \i in {2,3,4,5}
\fill [dgrey] 
       ({(\m-1)*\u+\i*\u},{(\n-1)*\u-5*\u}) rectangle  ({(\m-1)*\u+\i*\u+\u},{(\n-1)*\u-5*\u+\u});
       
 \foreach \i in {6,7,8,0}
\fill [dgrey] 
       ({(\m-1)*\u+\i*\u},{(\n-1)*\u-6*\u}) rectangle  ({(\m-1)*\u+\i*\u+\u},{(\n-1)*\u-6*\u+\u});

\foreach \i in {0,1,2,3}
\fill [dgrey] 
       (0,{(\m)*\u-\i*\u}) rectangle   ({0+\u},{(\m)*\u-\i*\u+\u});

\foreach \i in {4,5,6,7}
\fill [dgrey] 
        ({0+\u},{(\m)*\u-\i*\u}) rectangle   ({0+\u+\u},{(\m)*\u-\i*\u+\u});
       
\foreach \i in {8,0,1,2}
\fill [dgrey] 
        ({0+2*\u},{(\m)*\u-\i*\u}) rectangle   ({0+\u+2*\u},{(\m)*\u-\i*\u+\u});

\foreach \i in {3,4,5,6}
\fill [dgrey] 
        ({0+3*\u},{(\m)*\u-\i*\u}) rectangle   ({0+\u+3*\u},{(\m)*\u-\i*\u+\u});
        
\foreach \i in {7,8,0,1}
\fill [dgrey] 
           ({0+4*\u},{(\m)*\u-\i*\u}) rectangle   ({0+\u+4*\u},{(\m)*\u-\i*\u+\u});

\foreach \i in {2,3,4,5}
\fill [dgrey] 
           ({0+5*\u},{(\m)*\u-\i*\u}) rectangle   ({0+\u+5*\u},{(\m)*\u-\i*\u+\u});
       
 \foreach \i in {6,7,8,0}
\fill [dgrey] 
           ({0+6*\u},{(\m)*\u-\i*\u}) rectangle   ({0+\u+6*\u},{(\m)*\u-\i*\u+\u});

\foreach \i in {0,...,\n}
    {\draw [black, line width=0.1] (\u*\i,0)--(\u*\i,{\u*\n});
    \draw [black, line width=0.1] (0,\u*\i)--({\u*\n}, \u*\i);}
\draw [black, line width=2.5] 
      (0,0) rectangle   ({(\m-1)*\u},{(\m+1)*\u});
\draw [black, line width=2.5] 
      (0,{(\m+1)*\u}) rectangle   ({(\m-1)*\u},{(\n*\u});
\draw [black, line width=2.5] 
     ({(\m-1)*\u},{(\m+1)*\u}) rectangle   ({\n*\u},{(\n*\u});
\draw [black, line width=2.5] 
     ({(\m-1)*\u},0) rectangle   ({\n*\u},{((\m+1)*\u});
\node [below] at ({\n*\u},0) {$B_{22}$}; 
\node [below] at ({0},0) {$B_{21}$};   
\node [above] at ({\n*\u}, {\n*\u}) {$B_{12}$}; 
\node [above] at ({0}, {\n*\u}) {$B_{11}$};     

\end{tikzpicture}
\end{center}
\caption{An example of the construction of $B$ for 
$m=8$, $k=1$, $c=3$, $d=4$, which correspond to $b=11$ and $a=23$, the largest one with $m = 8$ for which $k=1$. The first column of $B_{12}$ is the only one with $4$ 1s, so the corresponding diagonal element of 
$B_{22}$ switches to 0. Points in the $a$-dominating set are represented by dark squares.}
\label{fig:domination}
\end{figure}

The next step is to 
use Lemma~\ref{lemma-ab-fill} (with $a=d$, $m = m-k$, and $n = m+k$), which provides 
a configuration in $B_{12}$ with the required row sums, and such that
each column sum is between $\lfloor (m-k)d/(m+k)\rfloor\ge c$ 
(by (\ref{eq-consistency})) and $\lceil (m-k)d/(m+k)\rceil\le c+1$
(by (\ref{eq-consistency-1})). 
Note that this last inequality implies that $c\ge 0$, as promised.
Recall also that, since $B$ is symmetric, $B_{21}$ is uniquely determined by $B_{12}$.

In the final step, we eliminate some 1s from $B_{22}$. 
Namely, for every column $j$ of $B$ which has $(c+1)$ 1s in $B_{12}$, the row $j$ of $B$ also has $(c+1)$ 1s in $B_{21}$. 
In this case, we replace the $1$ at the diagonal position $jj$ of $B$ (which must fall in $B_{22}$) by 0. 
It is clear from the construction that we have produced  an $a$-dominating set with claimed row and 
column sums. 
See Figure~\ref{fig:domination} for an example.

We have therefore established the $\le$ part in (\ref{eq-boot-odd-a>n}). 
In particular, this implies that $\gamma<n(b+1)$. 
We now proceed to verify the $\ge$ part. 
Recall the expression $\nu(s)$ from  Lemma~\ref{lemma-a>n-lower-bound}.
Note that in our case, the interval $\left[\lceil an/2\rceil, \lceil a/2\rceil n\right]$ from the lower bound given by  Lemma~\ref{lemma-a>n-lower-bound} equals to the interval $[bn+n/2, bn+n]$.
However, since $\gamma<n(b+1)$, the right endpoint of the interval can be safely excluded from consideration.
We next observe that $\nu$ is decreasing for $s\in [bn+n/2, bn+n-1]$. 
Indeed, on this interval $\lfloor s/n\rfloor=b$ and so 
$\nu$ is a linear function with coefficient 
$4b+2-a-2n=a-2n<0$. 
We therefore only need to check that $\nu(bn+n/2+k-1)>0$. 
After 
some algebra, we get 
\begin{equation*}
\nu(bn+n/2+k-1)=\frac52 n-a-(2n-a)k=(2n-a)\left(1+\frac{n}{2(2n-a)}-k\right)>0,
\end{equation*}
which establishes (\ref{eq-boot-odd-a>n}).
\end{proof}

\begin{proof}[Proof of Theorem~\ref{thm-boot} when $a>n$, and both $a$ and $n$ are odd] 
Now write $a=2b+1$ and $n=2m+1$. The argument is quite 
similar to the previous case, but differs in the details, so we go through it again. Now 
$$
B=\begin{bmatrix}
 B_{11}^{(m-k)\times (m-k)}
 &B_{12}^{(m-k)\times (m+k+1)}\\
 B_{21}^{(m+k+1)\times (m-k)}
 &B_{22}^{(m+k+1)\times (m+k+1)}
\end{bmatrix},
$$
and $B$ is symmetric, all entries of $B_{11}$ are $1$, the first $m-k$ rows each have $b$ 1s, and remaining $m+k+1$ rows each have $b+1$ 1s. 
Thus, the number of 1s is $(2m+1)b+m+1+k$, so that again
$k$ is the correction term. Domination again follows from 
Lemma~\ref{lemma-domination-check}. Further, now $c=b-m-k$, $d=b-m+k$, and the consistency requirement is
\begin{equation}\label{eq-consistency-odd}
    (m-k)d\ge (m+k+1)c, 
\end{equation}
which is equivalent to
$$
a(2k+1)\le n(4k+1), 
$$
and again we can solve the inequality for $k$ to get the correction term as the smallest $k$ that satisfies it. The construction is concluded as 
for even $n$, provided 
\begin{equation}\label{eq-consistency-2}
   (m-k)d\le (m+k+1)(c+1), 
\end{equation}
which is equivalent to 
$$
b\ge \frac{2k(n-1)}{2k+1}-\frac{1}{2k+1}. 
$$
By minimality of $k$, 
$$
(2b+1)(2k-1)>n(4k-3),
$$
which gives a lower bound on $b$, 
$$
b>\frac{4k-3}{2(2k-1)}n -\frac12, 
$$
which, after a bit of algebra, shows that (\ref{eq-consistency-2}) is satisfied if $k\ge 2$. 
We therefore need to address the case $k=1$ separately.
First, note that, since $d = c+2k$, for $k = 1$ the inequality  (\ref{eq-consistency-2}) is satisfied if we replace 
the factor $(c+1)$ by $(c+2)$. 
It follows that each column has at most 
$c+2$ 1s, and by construction in Lemma~\ref{lemma-ab-fill}, when a column 
does have $c+2$ 1s, all columns have at least $c+1$ 1s. 
Suppose first that there is a single column (of $B_{12}$) with $c+2$ 1s.
Then, we may assume it is the first one and we make the $2\times 2$ block in the upper left corner of $B_{22}$ to be 
$$
\begin{matrix}
    0 & 0\\
    0 & 1
\end{matrix}
$$
and put 0s on all diagonals of $B_{22}$ outside of this block. 
Suppose now that the number of columns with $c+2$ 1s is $\ell\ge 2$. 
Then, the $\ell\times \ell$ block in the upper corner of $B_{22}$ is a matrix with exactly $\ell-2$ 1s in its every row and column.
Such a matrix is the $2\times 2$ matrix with all entries equal to $0$ if $\ell = 2$, while for $\ell \ge 3$, we can take the matrix of $1$s minus the adjacency matrix of the cycle graph on $\ell$ vertices.
The diagonals of $B_{22}$ outside of this block are switched to $0$. 
This again results in the matrix $B$ with required 
row and column counts. 

Again, it remains to verify the $\ge$ part in (\ref{eq-boot-odd-a>n}). 
Exactly as for even $n$, we verify that $\nu$ from 
Lemma~\ref{lemma-a>n-lower-bound} is decreasing on the same
interval so that we only need to check that \hbox{$\nu(bn+(n+1)/2+k-1)>0$}. 
Now this yields
$$
\nu(bn+(n+1)/2+k-1)=\frac 32n-\frac12 a-(2n-a)k=
(2n-a)\left(\frac12+\frac{n}{2(2n-a)}-k\right)>0,
$$
finishing the proof of (\ref{eq-boot-odd-a>n}) for odd $n$. 
\end{proof}

\section{A $3$-approximation to Young domination}\label{sec-3-approx}

The first step in our proof of Theorem~\ref{thm-3-approx} is to compare $\gamma^L$  to one of its enhanced
versions; another version is used in Section~\ref{sec-C-approx}. In this comparison, we convert points in the initial 
set $A$ used in the standard dynamics into enhancements as indicated 
by Lemma~\ref{lemma-reg-enh}. The following definition accounts for the fact that the resulting enhanced dynamics may not cover points in~$A$. 

\begin{equation*}
\begin{aligned}
\gammap^L=\gammap^L(\Z,m,n)=\min\{|\vec r|+|\vec c|+|R_{n,m}\setminus \cY|\colon\
&\vec r\in \bZ_+^m, \vec c\in \bZ_+^n\text{ are nonincreasing,}\\
&\cY\subseteq R_{n,m}\text{ is a Young diagram, and,}\\
&\text{ for $\cT$ given by }(\Z,\vec r, \vec c),\text{ } \cT^L(\emptyset)=\cY\}.
\end{aligned}
\end{equation*}
As usual, we denote $\gammap^1$ simply by $\gammap$.

The inequalities in the next lemma follow from Lemma~\ref{lemma-reg-enh} and a repeated application of Lemma~\ref{lemma-enh-reg}.

\begin{lemma}\label{lem:approx-gamma} For any zero-set $\Z$ and positive integers $m,n$,
\[
\frac{1}{3} \,\gammap^L\leq \gamma^L\leq L\,\gammap^L.
\]
\end{lemma}

\begin{proof}
Fix a smallest set $A\subseteq R_{n,m}$ that occupies $R_{n,m}$ in $L$ steps and order the rows and columns of $R_{n,m}$ so that the row sums $\vec r=(r_0,r_1,\ldots)$ and column sums $\vec c=(c_0,c_1,\ldots)$ of $A$ are non-increasing. 
Starting from $\emptyset$, the dynamics given by $(\Z,\vec r, \vec c)$  occupies in $L$ steps a Young diagram $\cY\subseteq R_{n,m}$ (by Lemma~\ref{lem-zeroyoung}) such that $R_{n,m}\setminus \cY\subseteq A$ (by Lemma~\ref{lemma-reg-enh}).
Then $\gammap^L\le |\vec r|+|\vec c|+|R_{n,m}\setminus \cY|\le 3|A| = 3\gamma^L$, which proves the left inequality.

To prove the right inequality, assume that nonincreasing
$\vec r$, $\vec c$,  with 
$\vec r\in \bZ_+^m, \vec c\in \bZ_+^n$, define, 
together with $\cZ$, the enhanced dynamics $\widehat\cT$, 
and thus also give $\cY=\widehat\cT^L(\emptyset)$. Assume 
also that these enhancements realize $\widehat \gamma^L$. 
Let $\cT$ be the regular 
dynamics given by $\cZ$. 

Recall the enhancement operation defined on sets 
before Lemma~\ref{lemma-enh-reg}. Inductively define the following sets: 
$B_0=\widehat \emptyset$, $B_{k}=\widehat{\cT(B_{k-1})}$, for all $k=1,\ldots, L-1$. We first prove by induction that 
\begin{equation}\label{eq-L1}
 \widehat \cT^k(\emptyset)\subseteq \cT(B_{k-1}),   
\end{equation} 
for $k=1,\ldots, L$. By Lemma~\ref{lemma-enh-reg}, this 
holds for $k=1$. 
Assume now $2\le k\le L$ and that (\ref{eq-L1}) 
holds for $k-1$. Then 
$$
\widehat \cT^k(\emptyset)
=\widehat \cT(\widehat \cT^{k-1}(\emptyset))
\subseteq \widehat\cT( \cT(B_{k-2}))
\subseteq \cT( \widehat{\cT(B_{k-2})})=\cT(B_{k-1}),
$$
where we used, in order, the definition of $\widehat \cT^k$, the induction hypothesis and the monotonicity of $\widehat \cT$, Lemma~\ref{lemma-enh-reg}, and the definition ob $B_k$. This 
establishes (\ref{eq-L1}). 

Define 
$$
A=(R_{n,m}\setminus \cY) \cup B_0\cup \bigcup_{k=1}^{L-1} (B_k\setminus \cT(B_{k-1})) .
$$
We now prove by induction that 
\begin{equation}\label{eq-L2}
\cT(B_{k-1})\subseteq \cT^k(A), 
\end{equation} 
for $k=1,\ldots, L$. 
For $k=1$ this is true as $B_0\subseteq A$ and $\cT$ is monotone. Assume 
now that $k\ge 2$ and that (\ref{eq-L2}) is true for $k-1$, i.e., that
$\cT(B_{k-2})\subseteq \cT^{k-1}(A)$.  As we have 
$B_{k-1}\setminus \cT(B_{k-2})\subseteq A\subseteq \cT^{k-1}(A)$, it 
follows that $B_{k-1}\subseteq \cT^{k-1}(A)$ and so 
$\cT(B_{k-1})\subseteq \cT^{k}(A)$, establishing (\ref{eq-L2}).

It immediately follows from (\ref{eq-L1}) and (\ref{eq-L2}) that 
$
\cY=\widehat \cT^L(\emptyset)\subseteq \cT^L(A). 
$
Clearly, 
$$R_{n,m}\setminus \cY\subseteq A\subseteq \cT^L(A).$$ 
We thus conclude that  $\cT^L(A)=R_{n,m}$ and, hence, $\gamma^L\le |A|$. 
Finally, we estimate the size of $A$.
Observe that $|B_0|\le |\vec r|+|\vec c|$ and 
$|B_k\setminus \cT(B_{k-1})|\le |\vec r|+|\vec c|$
for all $k=1,\ldots L-1$. It follows that 
$$
|A|\le |R_{n,m}\setminus \cY|+L\,(|\vec r|+|\vec c|)
\le L\widehat \gamma^L, 
$$
and therefore, $\gamma^L\le |A|\le  L\,\widehat\gamma^L$, as desired.
 \end{proof}
 
	\begin{theorem}\label{thm-alg-gammap}
There exists an algorithm polynomial in $n$ that takes as input $m,n\in \bN$ with $m\le n$ and a zero-set $\Z$ with $\Z\subseteq R_{n,m}$, and computes the value of $\gammap = \gammap(\Z,m,n)$.
	\end{theorem}
    
\begin{proof}
In this proof, the enhancement vectors are always assumed to be nonincreasing.

We show how to compute $\gammap$ in polynomial time using a dynamic programming approach.
In order to explain it, we fix some notation and preliminary observations.
Consider a given zero-set $\Z$ and $m,n$ with $m\le n$ and $\Z\subseteq R_{n,m}$.
The set $\Z$ has a non-empty set of concave corners $\{z^0,z^1,\ldots, z^{p}\}$,
for some $p\ge 0$; we may assume that they are ordered so that,
writing $z^i = (a_{i},b_{p-i})$ for all $i= 0,1,\ldots, p$, we have
$a_0 = b_0 = 0$, $a_1<\ldots<a_p$, and $b_1<\ldots<b_p$ (see Fig.~\ref{fig:Young}).

\begin{figure}[h!]
\begin{center}
\includegraphics[width=0.6\textwidth]{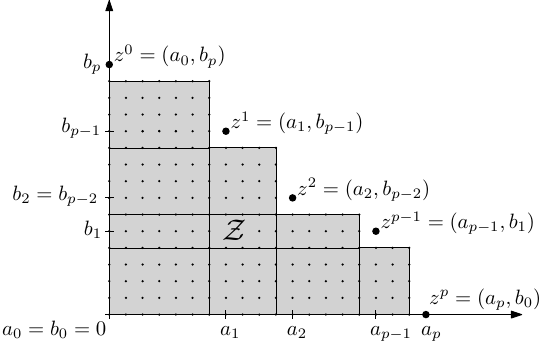}
\end{center}
\caption{A zero-set $\Z$ and its concave corners.}\label{fig:Young}
\end{figure}

Note that any pair of enhancements $\vec r\in \bZ_+^{m}$ and $\vec c\in \bZ_+^n$ defines a feasible solution for $\gammap$ by taking $\cY = \cT(\emptyset)$ for $\cT$ given by $(\Z,\vec r, \vec c)$. 
Indeed, it suffices to apply Lemma~\ref{lem-zeroyoung} to infer that $\cY$ is a Young diagram.

Feasible solutions for $\gammap$ can thus be identified with enhancement pairs $(\vec r,\vec c)$. Any such pair defines
a set of points $P_{\vec r,\vec c} := \{(r_u,c_v)\colon 0\le u\le m-1, 0\le v\le n-1\}$ in $\bZ_+^2$.
Let $Q_{\vec r,\vec c} = P_{\vec r,\vec c}\setminus \cZ$. We claim that in every optimal solution $(\vec r, \vec c)$ for $\gammap$,
the set of minimal points in $Q_{\vec r,\vec c}$, under the natural partial order on $\bZ_+^2$, is a subset of
the set $\{z^0,\ldots, z^p\}$ of concave corners of $\cZ$. Indeed, suppose
this is not the case and that there exists an optimal solution $(\vec r, \vec c)$ for $\gammap$ such that
some minimal point, say $(r_s,c_t)$, of $Q_{\vec r,\vec c}$ is not a concave corner of $\cZ$. Then, there is a concave corner $(a_q,b_{p-q})$ that is strictly below $(r_s, c_t)$, say $a_q<r_s$ and $b_{p-q}\le c_t$.
By optimality and the definition of $\gammap$, we have
$\gammap(\Z,m,n)=|\vec r|+|\vec c|+mn-|\{(v,u)\in R_{n,m}\colon (r_u, c_v) \notin \Z\}|$.
Since the condition $(r_u, c_v) \notin \Z$ is equivalent to the condition that
$(a_{j}, b_{p-j})\le (r_u,c_v)$ for some $j\in [0,p]$ (cf.~Fig.~\ref{fig:Young}),
replacing $\vec r$ with row enhancement $\vec r\,'$ defined by
$$r'_u = \left\{
           \begin{array}{ll}
             a_q, & \hbox{if $u = s$;} \\
             r_u, & \hbox{otherwise,}
           \end{array}
         \right.$$
results in a feasible solution $(\vec r\,', \vec c)$ for $\gammap$
such that $(r_u, c_v) \notin \Z$ if and only if $(r_u', c_v) \notin \Z$.
This implies that the objective function value at $(\vec r\,', \vec c)$ is smaller than the one at
$(\vec r, \vec c)$, contradicting the optimality of $(\vec r, \vec c)$.

The above observation implies that in order to compute the value of $\gammap$, we may restrict our attention to enhancement pairs $(\vec r,\vec c)$
such that $r_u\in \{a_0,\ldots, a_p\}$ for all $u\in [0,m-1]$ and $c_v\in \{b_0,\ldots, b_p\}$ for all $v\in [0,n-1]$.
We call such enhancement pairs \emph{tame}. Any tame enhancement pair $(\vec r,\vec c)$ can be represented by a pair of vectors $\vec x = (x_0,\ldots, x_p)\in \mathbb{Z}_+^{p+1}$ and $\vec y = (y_0,\ldots, y_p)\in \mathbb{Z}_+^{p+1}$
given by
\begin{eqnarray*}
  x_i &=& |\{u\colon  0\le u\le m-1, r_u = a_i\}|\,,\\
  y_j &=& |\{v\colon  0\le v\le n-1, c_v = b_{p-j}\}|\,.
\end{eqnarray*}
Clearly, we have $|\vec x| = m$ and $|\vec y| = n$, and, since enhancements are assumed to be nonincreasing sequences, any pair of vectors $(\vec x, \vec y)\in (\mathbb{Z}_+^{p+1})^2$ with $|\vec x| = m$ and $|\vec y| = n$ represents a unique tame enhancement pair.

Given such a vector pair $(\vec x,\vec y)$, the tame enhacement pair $(\vec r,\vec c)$ represented by
$(\vec x,\vec y)$ satisfies $|\vec r| = \sum_{i = 0}^pa_ix_i$ and $|\vec c| = \sum_{j = 0}^pb_{p-j}y_j$.
Moreover, letting $\cY = \cT(\emptyset)$ for $\cT$ given by $(\Z,\vec r, \vec c)$, we have:
\begin{eqnarray*}
|R_{n,m}\setminus \cY| &=& |\{(v,u)\in R_{n,m}\colon (r_u, c_v) \in \Z\}|~~~~\text{(by the definition of $\cY$)}\\
&=& \sum\{x_iy_j\colon 0\le i\le p\,,\,0\le j\le p\,,\,(a_i, b_{p-j}) \in \Z\}~~~~\text{(by the definition of $(\vec x,\vec y)$)}\\
&=& \sum_{0\le i<j\le p}x_iy_j~~~~\text{(since $(a_i,b_{p-j})\in \Z$ if and only if $i<j$, cf.~Fig.~\ref{fig:Young})}\,.
\end{eqnarray*}
Therefore, we can redefine $\gammap(\Z,m,n)$ equivalently in terms of $(\vec x,\vec y)$ as follows:
   \begin{equation}\label{eq:gamma_1'-alternative-form}
   \begin{aligned}
&\gammap(\Z,m,n)
\\&=\min\left\{\sum_{i = 0}^pa_ix_i+ \sum_{j = 0}^pb_{p-j}y_j+ \sum_{0\le i<j\le p}x_iy_j\colon 
	\vec x, \vec y \in \mathbb{Z}_+^{p+1}\,,\,|\vec x| = m\,,\,|\vec y| = n\right\}\,.
 \end{aligned}
\end{equation}

We now have everything ready to explain our algorithm for computing $\gammap(\cZ,m,n)$.
Given $\Z$, $m$, $n$, the $a_i$s and the $b_j$s as above, we use a dynamic programming approach to compute optimal values to polynomially many instances of a problem generalizing $\gammap$. 
To define the set of inputs for this more general problem, we let, 
for $q\in[0,p]$, 
$$
\cV_q=\{(k,\ell,C)\in\bZ_+^3\colon  k\in[0,m], \ell\in [0,n], 
C\in [0,(p-q)n]\}.
$$
The task is then to compute, for all $q\in[0,p]$, and all 
$(k,\ell,C)\in\cV_q$, the value $f_q(k,\ell,C)$ given by the following expression:
 \begin{equation*}
   \begin{aligned}
   &f_q(k,\ell,C)\\& = \min\left\{\sum_{i = 0}^q(a_i+C)x_i+ \sum_{j = 0}^qb_{p-j}y_j+ \sum_{0\le i<j\le q}x_iy_j\colon 
	\vec x, \vec y \in \mathbb{Z}_+^{q+1}\,,\,|\vec x| = k\,,\,|\vec y| = \ell\right\}\,.
 \end{aligned}
 \end{equation*}
Computing all the $f_q(k,\ell,C)$ values suffices, since the value of
$\gammap(\Z,m,n)$ is given by $\gammap(\Z,m,n) = f_p(m,n,0)$.

The values $f_q(k,\ell,C)$ for all $q\in [0,p]$ and $(k,\ell,C)\in\cV_q$ can be computed recursively as follows. 

\noindent\emph{Step 1}. Compute $f_0(k,\ell,C)$ for all $(k,\ell,C)\in\cV_0$. 

For all $(k,\ell,C)\in\cV_0$, we have
  $$f_0(k,\ell,C) = \min\left\{(a_0+C)x_0+ b_{p}y_0\colon  \vec x, \vec y \in \mathbb{Z}_+\,,\,|\vec x| = k\,,\,|\vec y| = \ell\right\}\,.$$
  As the only feasible solution to the above problem is $(\vec x,\vec y) = (k,\ell)$, we have
  $f_0(k,\ell,C) = (a_0+C)k+ b_{p}\ell$.
  
   \noindent\emph{Step 2}. Suppose that $q\ge 1$ and that we have already computed the values $f_{q-1}(k',\ell',C')$ for all $(k',\ell',C')\in\cV_{q-1}$. 
   Then we compute all the values $f_q(k,\ell,C)$ for $(k,\ell,C)\in \cV_q$.
  
The values of $k$, $\ell$, $C$, $q$ are considered fixed throughout Step 2, so the dependency of various quantities on them is suppressed from the notation. 
We partition the set ${\cal D}$ of feasible solutions $(\vec x,\vec y)$ for $f_q(k,\ell,C)$ according to the values of the last coordinates, $x_q$ and $y_q$: for $A\in [0,k]$ and $B\in [0,\ell]$, we set ${\cal D}_{A,B} = \{(\vec x,\vec y)\in {\cal D}\colon x_q = A, y_q = B\}$.
  Then
    \begin{equation}\label{eq:fKLCq}
f_q(k,\ell,C) = \min_{\substack{0\le A\le k\\0\le B\le \ell}}g(A,B),
  \end{equation}
  where
  $$g(A,B) = \min\left\{\sum_{i = 0}^q(a_i+C)x_i+ \sum_{j = 0}^qb_{p-j}y_j+ \sum_{0\le i<j\le q}x_iy_j\colon  (\vec x,\vec y)\in {\cal D}_{A,B}\right\}\,.$$
 By~\eqref{eq:fKLCq}, it follows that in order to compute the value $f_q(k,\ell,C)$ it suffices to compute the $\cO(|k||\ell|)$ values  $g(A,B)$.

Since every feasible solution $(\vec x,\vec y)\in {\cal D}_{A,B}$ for $g(A,B)$
satisfies $x_q = A$ and $y_q = B$, the objective function value simplifies to
\begin{equation*}
\begin{aligned}
   & \sum_{i = 0}^{q-1}(a_i+C)x_i+ (a_q+C)\cdot A + \sum_{j = 0}^{q-1}b_{p-j}y_j + b_{p-q} \cdot B + \sum_{0\le i<j\le q-1}x_iy_j + \left(\sum_{i = 0}^{q-1}x_i\right)\cdot B \\
  &= \left(\sum_{i = 0}^{q-1}(a_i+B+C)x_i+ \sum_{j = 0}^{q-1}b_{p-j}y_j + \sum_{0\le i<j\le q-1}x_iy_j\right) + ((a_q+C)\cdot A + b_{p-q} \cdot B)\,.
  \end{aligned}
\end{equation*}
Therefore, by considering the vectors $\vec x\,'$ and $\vec y\,'$ obtained from $\vec x$ and $\vec y$ by removing the last coordinates,
$\vec x\,' = (x_0,\ldots, x_{q-1})$ and $\vec y\,' = (y_0,\ldots, y_{q-1})$, we infer that
\begin{equation*}
	\begin{aligned}
&g(A,B)\\ &= \left((a_q+C)\cdot A + b_{p-q} \cdot B\right) &+ \min\Bigg\{\sum_{i = 0}^{q-1}(a_i+B+C)x'_i+ \sum_{j = 0}^{q-1}b_{p-j}y'_j+ \sum_{0\le i<j\le q-1}x'_iy'_j\colon\\&&
\vec x\,',\vec y\,'\in \mathbb{Z}_+^{q-1}\,,\, |\vec x\,'| = k-A\,,\,|\vec y\,'| = \ell-B\Bigg\}\,,
	\end{aligned}
	\end{equation*}
which simplifies further to
\begin{equation}\label{eq:gAB}
g(A,B) = ((a_q+C)\cdot A + b_{p-q} \cdot B) + f_{q-1}(k-A,\ell-B,B+C)\,.	
\end{equation}
Note that since $C\le (p-q)n$ and $B\le \ell\le n$, we have $B+C\le (p-q+1)n$, hence, by our assumption,
the value of $f_{q-1}(k-A,\ell-B,B+C)$ was already computed.
Using~\eqref{eq:fKLCq} and~\eqref{eq:gAB}, this shows that the value $f_q(k,\ell,C)$ can be computed in time $\cO(|k||\ell|)$ from the previously computed values.

The number of values $f_q(k,\ell,C)$ 
for all $q\in [0,p]$ and $(k,\ell,C)\in \cV_q$ 
is $\cO(mn^2p^2)$ and they can 
therefore all be computed in time $\cO(m^2n^3p^2)$.
Recalling that $\gammap(\Z,m,n) = f_p(m,n,0)$,
this completes the proof.
\end{proof}

\begin{proof}[Proof of Theorem~\ref{thm-3-approx}]
The statement follows immediately from Lemma~\ref{lem:approx-gamma} and Theorem~\ref{thm-alg-gammap}.
\end{proof}
			
\section{A constant approximation to Young domination with fixed latency}\label{sec-C-approx} 

In this section we make use of a simpler version of enhanced
domination numbers, namely the minimum norm of enhancements such that the empty set spans in $L$ steps. 
While it gives a less accurate approximation,
it can be used to handle arbitrary latency. Let
\begin{equation*}
\begin{aligned}
\gammad^L=\gammad^L(\Z,m,n)=
\min\{|\vec r|+|\vec c|\colon
\text{ for $\cT$ given by }(\Z,\vec r, \vec c),\text{ } \cT^L(\emptyset)=R_{n,m}\}.
\end{aligned}
\end{equation*}

\begin{lemma}\label{lem:approx-gamma_k'-gamma_k''}
For every $L<\infty$ there is a constant $\delta>0$ such that
$$\delta\gammad^L\le \gammap^L\le \gammad^L.$$
\end{lemma}

\begin{proof}
Without loss of generality we may assume that $m\le n$.

The right inequality is immediate from the definitions.
The proof of the left inequality is divided into
three cases.

\noindent{\it Case 1\/}: $m\le 100$.

\noindent Let $a$ be the length of the longest row of $\Z$.
Then, by Lemma~\ref{lem:approx-gamma}, 
$$L\,\gammap^L\ge\gamma^L
\ge \gamma^\infty\ge a, $$ as a set $A$ consisting of any $a-1$ fully occupied columns is a fixed point: $\cT(A)=A$. 
Also, $\gammad^L\le \gammad^1\le ma\le 100a$, as 
constant $\vr$ with $m$ components equal to $a$
spans in one step. 
Therefore $\gammad^L\le 100L\,\gammap^L$.

For the remaining two cases, choose $s$ to be the minimal positive integer such that \hbox{$(s, \lfloor ms/n\rfloor)\notin\cZ$}.

\noindent{\it Case 2\/}: $m>100$, $s\le n/8$.

\noindent The constant enhancements $\vr=(s,\ldots,s)\in \bZ_+^m$, 
$\vc=(\lfloor ms/n\rfloor, \dots, \lfloor ms/n\rfloor)\in\bZ_+^n$
span in one step, so we have $\cY=R_{n,m}$, and therefore,
$$\gammap\le ms+n \floor{ms/n}\le 2ms,$$
and so we may restrict to $\cY$
with $|R_{n,m}\setminus \cY|\le 2ms$ in our objective function
for $\gammap^L$. 
Assume that some enhancement pair $(\vr,\vc)$, together with a Young diagram $\cY$ that satisfies this restriction, realize $\gammap^L$. 
Then consider the largest $t$ such that 
$$
(m-\lceil mt/n\rceil, n-t)\in R_{n,m}\setminus \cY.
$$
Then 
$$
t\cdot \lceil mt/n\rceil\le |R_{n,m}\setminus \cY|\le 2ms
$$
and so $t^2\le 2ns\le n^2/4$ and $t\le n/2$. As 
$$
(m-\lceil m(t+1)/n\rceil, n-t-1)\in\cY, 
$$
$\cY$ includes a rectangle whose number of columns is 
$n-t\ge n/2$ and the number of rows at least 
$m-mt/n-1\ge m/2-1\ge m/3$. 

It follows that the dynamics with enhancements $\vr$ and 
$\vc$ occupy in time $L$ a rectangle with $\lceil m/3\rceil$ rows and $\lceil n/2\rceil$ columns.
By permuting the enhancements we may assume this occupied rectangle is placed anywhere in our universe $R_{n,m}$. Clearly, we can cover $R_{n,m}$ with six properly placed such rectangles, each with its own enhancement vector. 
Summing the six row and six column enhancement vectors therefore produces row and column enhancements vectors $\vr_0$ and $\vc_0$ that occupy the entire $R_{n,m}$ at time $L$, and $|\vr_0|\le 6|\vr|$ and 
$|\vc_0|\le 6|\vc|$. Therefore, in this case, 
$$\gammad^L\le |\vr_0|+|\vc_0|\le 6(|\vr|+|\vc|)\le 6\gammap^L.$$

\noindent{\it Case 3\/}: $m>100$, $s> n/8$.

\noindent Using Lemma~\ref{lem:approx-gamma} and \cite[Theorems 1 and 2]{GSS},
\begin{equation*}
 \begin{aligned}
 L\,\gammap^L&\ge 
\gamma^L
\ge \gamma^\infty
\ge
|\cZ|/4
\\&\ge(s-1)(m(s-1)/n-1)/4
\\&\ge
(n/9)\cdot (m/9-1)/4
\\&\ge mn/360
 \end{aligned}   
\end{equation*}
as $n\ge m>100$. 
On the other hand, we trivially have $\gammad^L\le \gammad^1\le mn$.

From the three cases, we see that we can take 
$\delta=1/(360 L)$.
\end{proof}

A polynomial algorithm for $\gammad^L$ will be a consequence of the following lemma.

 \begin{lemma}\label{lem:poly-gamma''_k} For any $L\ge 1$,
 every minimal $\vr$ and $\vc$ for $\gammad^L$
 each have at most $2^{L}$ different values.
 \end{lemma}

 \begin{proof}
 Assume that $\vr$ and $\vc$ are some minimal enhancements
 for  $\gammad^L$ and that they are both nonincreasing. 
 These determine the enhanced dynamics $\cT$ and will be fixed for the rest of the
 proof.
 
A Young diagram $\cY$ is a {\it trigger\/} for a point
$(a,b)\in R_{n,m}$, if one of the following holds:
\begin{itemize}
\item $\cY=R_{a+1,b+1}$; or
\item $\cY$ has exactly two convex corners: one at $(a,b_1)$ for
$0\le b_1< b$ and one at $(a_1,b)$ for
$0\le a_1< a$ and
$(r_b+a_1+1,c_a+b_1+1)$ is a concave corner of $\cZ$; or
\item $\cY$ has exactly one convex corner at $(a_1,b)$ for
$0\le a_1< a$ and
$(r_b+a_1+1,c_a)$ is a concave corner of $\cZ$; or
\item $\cY$ has exactly one convex corner at $(a,b_1)$ for
$0\le b_1< b$ and $(r_b,c_a+b_1+1)$
is a concave corner of $\cZ$; or
\item $\cY=\emptyset$ and
$(r_b, c_a)$ is a concave corner of $\cZ$.
\end{itemize}
Observe that, if $\cY$ is a trigger for $(a,b)$, then
$(a,b)\in\cT(\cY)$. Conversely, if $(a,b)\in \cT(\cY')$, for some 
Young diagram $\cY'$, then there is a trigger 
$\cY\subseteq \cY'$ for $(a,b)$. Indeed, the five items 
in the definition of the trigger correspond to the 
cases when $(a,b)\in \cY'$, or $(a,b)\in \cT(\cY')\setminus \cY'$ and uses, in addition to the enhancements, both $\cN^r(a,b)\cap \cY'$ and $\cN^c(a,b)\cap \cY'$, one of them, or neither, to become occupied. See Figure~\ref{fig:Ltrigger} for an example.

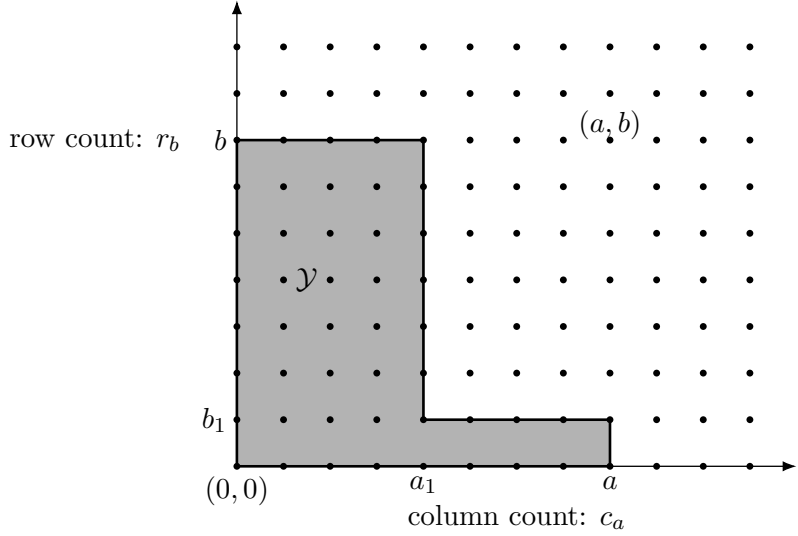
\begin{figure}[h!]
\begin{center}
\definecolor{lgrey}{rgb}{0.7, 0.7, 0.7}
\hskip-2cm
\begin{tikzpicture}
\def\u{0.62}
\fill[lgrey] (0,0) rectangle (4*\u,7*\u);
\fill[lgrey] (0,0) rectangle (8*\u,\u);
\draw [black, line width=1] 
      (0,0)--(0,7*\u)--(4*\u,7*\u)--(4*\u,\u)--(8*\u,\u)--(8*\u,0)--(0,0); 
\foreach \i in {0,...,11} 
 \foreach \j in {0,...,9} 
 \fill [black] (\u*\i, \u*\j) circle (1.3pt);
 \draw [-Latex,line width=0.5] (0,0) -- (12*\u,0);
 \draw [-Latex,line width=0.5] (0,0) -- (0,10*\u);
 \node [below] at (8*\u,0) {$a$};
 \node [below] at (4*\u,0) {$a_1$};
 \node [below] at (6*\u,-0.6*\u) {column count: $c_a$};
  \node [left] at (0,7*\u) {$b$};
  \node [left] at (0,1*\u) {$b_1$};
  \node [left] at (-\u,7*\u) {row count: $r_b$};
  \node at (1.5*\u,4*\u) {$\mathcal Y$};
  \node [above] at (8*\u,6.8*\u) {$(a,b)$};
  \node [below] at (0,0) {$(0,0)$};
\end{tikzpicture}
\end{center}
\caption{A trigger $\cY$ for
the point $(a,b)=(8,7)$. For this example, we
assume that $(r_7+5, c_8+2)$ is
a concave corner of $\cZ$. Vertices
that touch the shaded area comprise $\cY$.}
\label{fig:Ltrigger}
\end{figure}

Note also that any convex corner
of $\cY$ shares at least one coordinate with $(a,b)$,
and when $\cY$ has two corners, each shares a
different coordinate with $(a,b)$.

 We now define, by backward recursion, a sequence of Young diagrams
$\cY_{L}\supseteq\cY_{L-1}
\supseteq\ldots\supseteq \cY_0=\emptyset$
such that $\cY_i\subseteq\cT^i(\emptyset)$ for all $i=0,1,\ldots, L$.
We start with $\cY_L=R_{n,m}$.
The set $\cY_{L-1}'=\cT^{L-1}(\emptyset)$ of occupied sites
of the enhanced dynamics at time $L-1$ is
a Young diagram (Lemma~\ref{lem-zeroyoung}).
As the enhancements span
in $L$ steps, $(n-1,m-1)\in\cT(\cY_{L-1}')$ and
therefore there exists a trigger
$\cY_{L-1}\subseteq\cY_{L-1}'$ for
 $(n-1,m-1)$ (arbitrarily chosen in case it is not unique).

In general, assume that we have chosen $\cY_i$, for some $i\in\{1,\ldots,L\}$.
 Let $\cY_{i-1}'=\cT^{i-1}(\emptyset)$. 
 By the induction hypothesis, $\cY_i\subseteq\cT^i(\emptyset) = \cT(\cY_{i-1}')$.
In particular, for any convex corner $z\in \cY_i$, we have
$z\in \cT(\cY_{i-1}')$ and therefore there exists a trigger $\cY_{i-1,z}\subseteq\cY_{i-1}'$ for $z$. 
Now let
\begin{equation}\label{trigger-induction}
\cY_{i-1}=\bigcup_{\text{$z$ convex corner of $\cY_i$}}\cY_{i-1,z}.
\end{equation}
Then $\cY_{i-1}\subseteq \cY_i$ and $\cY_{i-1}
\subseteq\cT^{i-1}(\emptyset)$, which completes
the inductive step in the construction. Note that
$\cY_0\subseteq \cT^0(\emptyset)=\emptyset$.

For all $i = 0,1,\ldots, L$, let $Q_i$ be the set of all convex corners of $\cY_i$ and $Q=\bigcup_{i=0}^L Q_i$.
Let $\pi_x$ and $\pi_y$ be the projections onto the $x$-axis and the $y$-axis, respectively.
For all $j = 0,1,\ldots, L$, let $P_j = \bigcup_{i = j}^L \pi_x(Q_i)$.
By definition of a trigger and (\ref{trigger-induction}),
$$|P_j|\le 2|P_{j-1}|$$
for every $i=1,\ldots, L$, and therefore $|\pi_x(Q)|\le 2^{L}$.
Analogously, $|\pi_y(Q)|\le 2^{L}$.

\newcommand{\vrp}{\vr\,'}
\newcommand{\vcp}{\vc\,'}

 Now assume that $i_1<\ldots<i_A=n-1$ and $j_1< \ldots<j_B=m-1$,
 $A,B\le 2^{L}$ are exactly the elements of $\pi_x(Q)$ and
 $\pi_y(Q)$, respectively.
Assume that $\vrp$ and $\vcp$ are any nonincreasing enhancements
whose values agree at these indices, that is,
$c_{i_a}'=c_{i_a}$, $a=1,\ldots, A$,
and
$r_{j_b}'=r_{j_b}$,  $b=1\ldots, B$.
Denote the dynamics with these enhancements by $\widehat\cT$.
Then we claim that 
$\widehat\cT^L(\emptyset)=R_{n,m}$.

To verify the claim, we prove by induction on $i$ that
$\cY_i\subseteq \widehat \cT^i(\emptyset)$. This
is clearly true for $i=0$, and the induction
step follows from the construction of $\cY_i$.

As $\vr$ and $\vc$ are assumed minimal, the above claim
implies that any of their values must equal a value at
one of the above indices $i_a$ or $j_b$. Namely,
 for any $j=0,\ldots,m-1$,
 $r_j=r_{j_b}$ for the smallest $j_b\ge j$ and,
 for any $i=0,\ldots,n-1$,
 $c_i=c_{i_a}$ for the smallest $i_a\ge i$. As the number of 
 these indices is at most $2^{L}$, this concludes the proof.
\end{proof}

\begin{theorem} \label{thm-alg-gammad}
Assume $a,b\in \bN$, $2\le L<\infty$, and $\cZ\subseteq R_{a,b}\subseteq R_{n,m}$. 
Then $\gammad^L\le ab$ and $\gammad^L$ can be computed in time
$\cO((\min(m,ab)\cdot \min(n,ab))^{2^{L}+1}(ab)^{2^{L}})$.
\end{theorem}

\begin{proof}
First observe that there exist enhancements $\vr$ and $\vc$
with $|\vr|+|\vc|=ab$ that span at time~$2$: for example, we can let $r_0=\cdots=r_{b-1}=a$ and make all remaining enhancements~$0$ (in particular, $c_j = 0$ for all $j$). 
Therefore, since $L\ge 2$, we obtain that $\gammad^L\le ab$ and we may restrict our enhancements to those with at most $ab$ nonzero components. 

Let $m'=\min(m,ab+1)$, $n'=\min(n,ab+1)$. With our restriction, as we can assume that the enhancements are nonincreasing, we only 
need to choose   first $m'$ row, and first $n'$ column enhancements. 
By Lemma~\ref{lem:poly-gamma''_k}, it suffices to choose at most $2^{L}-1$ dividing points between different row enhancements, and the same for column enhancements. 
This can be done in $\cO((m'n')^{2^{L}-1})$ many ways.
Next, note that we can replace any row (resp.~column) enhancement  larger than $a$ by $a$ (resp.~larger than  $b$ by $b$) without affecting the spanning time. 
Therefore, for any fixed choices of the $2^{L}-1$ dividing points for row and column indices, the number of candidates for minimal enhancement vectors $\vr$ and $\vc$ is $\cO((a b)^{2^{L}})$. 
For every such candidate, we only need to run the dynamics on $R_{n',m'}$ to verify if it spans in $L$ steps (by part (3) of Lemma 2.3 in \cite{GPS}). For a single update, we need to compute row and column sums and check against every concave corner of $\cZ$.
Since $\cZ$ has at most $\min(a,b)$ concave corners, a single update can be computed in time $\cO(m'n'(m'+n'+\min(a,b))) = \cO((m'n')^2)$; therefore, since $L$ is fixed, $L$ steps can be computed  in time $\cO((m'n')^{2})$. 
Multiplying the number of candidates by the verification time for each gives the claimed upper bound. 
\end{proof}

\begin{proof}[Proof of Theorem~\ref{thm-const-approx}]
By Theorem~\ref{thm-alg-gammad}, 
there exists a constant $C>1$ such that
for every fixed $L$, there exists a $C$-approximation algorithm for $\gammad^L$ that runs in time polynomial in $ab$. 
Lemmas~\ref{lem:approx-gamma} and~\ref{lem:approx-gamma_k'-gamma_k''}  now finish the proof.
\end{proof}

\section{Comparison between different latencies}\label{sec-comp}

In the next two lemmas, we assume that $a$ and $b$ are the sizes of the longest row and column of $\cZ$, respectively.

\begin{lemma}\label{lemma-gamma-2}
If $m,n>ab$, then 
$$
\gamma^2(\cZ,m,n)= \min (bx+ay-xy), $$
where the minimum is over all concave corners $(x,y)$ of $\cZ$. 
\end{lemma}

\begin{proof}
Assume that $D$ is a set that realizes $\gamma^2$. As $R_{a,b}$
clearly occupies $R_{n,m}$ in two steps, $|D|\le ab$. By permuting 
rows and columns, we may assume that $D\subseteq [0,ab-1]\times [0,ab-1]$, and therefore the last row and the last column of $R_{n,m}$ contain no elements of $D$. 
In order for the point $(n-1,m-1)$ to become 
occupied at time $2$, there have to be at time $1$ some $x\in [0,a]$ occupied points
in the last row and some $y\in[0,b]$ occupied points in the last 
column, with $(x,y)$ a concave corner of $\cZ$. Each of the $x$ points must 
have at least $b$ points in $D$ in its column, while each of the 
$y$ points must have at least $a$ points in its row.

Thus, we may assume that the first $x$ columns each contain $b$ points 
in $D$; then, some $y$ rows each contain $a$ points in $D$, 
at least $a-x$ of them not within the first $x$ columns. Therefore, 
$\gamma^2=|D|\ge bx+(a-x)y$.

To show the opposite inequality, take $(x,y)$ to be the concave 
corner at which the minimum is achieved, then 
take 
$$
D=R_{a,y}\cup R_{x,b},
$$
which occupies all points on the first $x$ columns and first $y$ rows
of $R_{n,m}$ at time $1$ and then the entire $R_{n,m}$ at time $2$. 
 \end{proof}

\begin{lemma}\label{lemma-large-R}
    If $L\ge 2$ and $m,n>ab$, then $\gamma^L(\cZ,m,n)= \gamma^L(\cZ,\infty,\infty)$.
\end{lemma}
 
 \begin{proof}
Consider a set $D$ that realizes  $\gamma^L(\cZ,m,n)$. 
Then observe that every time a point $(i,j)$ outside $[0,ab-1]\times [0,ab-1]$
becomes occupied: if $i\le ab-1$, then the entire $i$th row gets occupied at the same time; if $j\le ab-1$, then the entire $j$th column gets occupied 
at the same time; and, therefore, when both $i,j\ge ab$, then the entire $R_{n,m}$ becomes occupied. 
As the argument is independent on the exact values of $m$ and $n$, provided that they are large enough, the claim remains true if we extend $m$ and $n$ to infinity.
\end{proof}

\begin{proof}[Proof of Theorem~\ref{thm-different-L}]
    The first statement is Lemma~\ref{lemma-large-R}. The second statement 
    follows from \cite[Theorem 1.1]{GPS} and \cite[Theorem 1]{GSS}. 
    We now address the three cases of the third statement, which can already be achieved with $m=n$; we assume so for the rest of the proof.

    When $L=0$ or $L=1$, (\ref{eq-different-L}) follows by taking 
    any fixed $\cZ$ and large $n$, but for concreteness we assume
    $\cZ$ is $T_a$ given by (\ref{eq-bootstrap}) for even $a\ll n$, so that
    $\gamma^1(\cZ,n,n)=an/2$ by (\ref{eq-boot-even-a}).
    Then $\gamma^0(\cZ,n,n)=n^2\gg an/2$ and  
    $\gamma^2(\cZ,n,n)\le a^2\ll an/2$, where the  inequality $\gamma^2(\cZ,n,n)\le a^2$ follows from the fact that the $a\times a$ square spans in two steps. 

    For $L=2$, consider 
    $$
    \cZ=[0,a-1]^2\cup (\{0\}\times [0,b-1])\cup ([0,b-1]\times \{0\}), 
    $$
    with $1\ll a\ll b\ll \sqrt n$. 
    Then we know by Lemma~\ref{lemma-gamma-2} 
    that $\gamma^2(\cZ,n,n)\ge ab$. On the other hand, if $\cZ$ is itself the initial 
    set, we get  
    $$
   \cT^1(\cZ)\supseteq ([0,a-1]\times [0,b-1])\cup ([0,b-1]\times [0,a-1]), 
    $$
    $$
   \cT^2(\cZ)\supseteq( [0,a-1]\times [0,n-1])\cup ([0,n-1]\times [0,a-1]), 
    $$
    and so $\cT^3(\cZ)=[0,n-1]^2$, that is, $\cZ$ is $\cZ$-dominating with latency $3$. Therefore, 
    $\gamma^3(\cZ)\le |\cZ|< a^2+2b\ll ab$.
\end{proof}

\section{Open problems and possible further directions}\label{sec-open}

The present work leaves open the following questions. 
\begin{enumerate}
    \item\label{open-problem-1} It is well known that computing the domination number of a graph is, in general, \textsf{NP}-hard.
    This remains true for distance-$L$ domination for any fixed finite $L$ (see, e.g.,~\cite{Hen}).
    Consequently, for any fixed finite $L$, the problem of computing $\gamma^L(T_1,G_1,G_2)$ is \textsf{NP}-hard (where $T_1 = \{(0,0)\}$), even if one of the two factors $G_1$ and $G_2$ is a one-vertex graph.
    (Note that the problem of computing $\gamma^\infty(T_1,G_1,G_2)$ is trivial.) 
    In particular, for any fixed finite $L$, the problem of computing $\gamma^L(\cZ,G_1,G_2)$ is \textsf{NP}-hard.
    However, these considerations tell nothing about the special case when both factors are complete, for any $L$ (finite or infinite).
    \begin{quote}
        What is the computational complexity of computing 
    $\gamma^L(\cZ,m,n)$?
    \end{quote}
    The complexity of the special case $L = m = n = \infty$ is the final open problem in~\cite{GSS}.
  
    \item \label{open-problem-2} Notice that if $\Pi$ is a Young domination (at latency $L$) on Hamming rectangles such that the Young diagram only has a constant number of concave corners, then $\Pi$ becomes \emph{sparse}: if the rectangle is of dimensions $m$ and $n$, and $k$ is the number of concave corners, then the number of possible instances is at most $\cO((mn)^k)$, which is a polynomial for every fixed $k$.
    Mahaney~\cite{Mah} proved that, if a sparse problem is \textsf{NP}-complete, then \textsf{P} = \textsf{NP}, which is unlikely. 
    This result says nothing about the general case (in which our approximation algorithms still apply), and it also leaves open the following problem.

    \begin{quote}
        Find an explicit polynomial-time algorithm for Young domination (at a fixed latency $L$) restricted to instances with constantly many concave corners.
    \end{quote}

    An analogous question can be asked for any variant of the problem in which we restrict the Young diagrams to any family that is polynomially-sized with respect to $mn$.  

      \item 
    Recall from \cite[Theorems 1 and 2]{GSS} that the value of $\gamma^{\infty}(\cZ,m,n)$ can be $4$-approximated in polynomial time, while for every finite $L$, Theorem~\ref{thm-const-approx} gives a polynomial-time algorithm that approximates $\gamma^L(\cZ,m,n)$ up to a constant factor in polynomial time, where both the approximation ratio and the exponent of the polynomial depend on $L$. 
    These results suggest the following questions.
    
    \begin{quote}
    Is there a constant $\delta>0$ such that for every $L$, there is an algorithm that approximates $\gamma^L(\cZ,m,n)$ up to factor $\delta$ in polynomial time? 
    \end{quote}
    
    \begin{quote}
    For finite $L$, is there an algorithm that approximates 
    $\gamma^L(\cZ,m,n)$ up to a constant factor 
    in polynomial time with powers independent of $L$? 
    \end{quote}

    \item Theorem~\ref{thm-different-L} showed that for $L = 0,1,2$, the Young domination number with latency $L$ is not bounded from above by any constant multiple of the Young domination number with any higher latency.
    This suggests the following question.
\begin{quote}
     Is (\ref{eq-different-L}) true for all $L\ge 0$?
\end{quote}   
    
\end{enumerate}
There are two natural directions for exploration of Young domination outside of Cartesian products of two graphs. For the first one, observe that our setting works on any graph with two-coloring of edges, i.e., on any ordered pair of two edge-disjoint graphs on the same vertex-set. 
The second direction is a multidimensional version. The zero-set is now a $d$-dimensional Young diagram, i.e., a downward-closed set $\cZ\subseteq \bZ_+^d$, and the dynamics proceeds on any graph with a $d$-coloring of edges, for example on a product graph $G_1\Box\cdots \Box G_d$ of $d$ graphs, with the natural coloring. 

\section*{Acknowledgments} 

The authors are grateful to Michael Lampis for helpful discussions.
JG was partially supported by the Slovenian Research and Innovation Agency research program P1-0285 and Simons Foundation Award \#709425. 
MK was partially supported by the Slovenian Research and Innovation Agency (research program P1-0383 and research project N1-0370).
MM was partially supported by the Slovenian Research and Innovation Agency 
(I0-0035, research program P1-0285 and research projects J1-60012, J1-70035, J1-70046, and N1-0370) and by the research program CogniCom (0013103) at the University of Primorska.

\end{document}